\documentclass[12pt,twoside,letter]{amsart}
\usepackage{mathrsfs}

\usepackage{txfonts}
\usepackage{bbm}

\usepackage{amsfonts}
\usepackage{amsmath}
\usepackage{amssymb}
\usepackage{wasysym}
\usepackage{psfrag}
\usepackage{graphics,epsfig,amsmath}  
\usepackage{comment}
\usepackage{xcolor}
\usepackage{enumerate}
\definecolor{mycolor}{rgb}{0.2, 0.1, 0.65}
\definecolor{gcolor}{rgb}{0.20,0.43,0.09}
\definecolor{3color}{rgb}{0.10,0.9,0.9}
\definecolor{4color}{rgb}{0.30,0.9,0.09}
\newenvironment{nota}{\medskip\noindent{\sc
Notation}.}{\goodbreak\medskip}
\newenvironment{notas}{\medskip\noindent{\sc
Notations}.}{\goodbreak\medskip}
\newtheorem{theorem}{Theorem}
\newtheorem{definition}{Definition}

\newenvironment{remks}{\noindent{\sc
Remarks}. }{\goodbreak\vskip10pt}
\newenvironment{ques}{\noindent{\sc
Question}. }{\goodbreak\vskip10pt}
\newtheorem{lemma}{Lemma}
\newtheorem{prop}{Proposition}
 
\newtheorem{cor}{Corollary}

\def\N{\mathbb{N}}
\def\R{\mathbb{R}}
\def\T{\mathbb{T}}
\def\cA{{\mathcal A}}
\def\cD{{\mathcal D}}
\def\cE{{\mathcal E}}
\def\cF{{\mathcal F}}
\def\cH{{\mathcal H}}
\def\cG{{\mathcal G}}
\def\cK{{\mathcal K}}
\def\cL{{\mathcal L}}
\def\cM{{\mathcal M}}
\def\cN{{\mathcal N}}
\def\cS{{\mathcal S}}
\def\cU{{\mathcal U}}
\def\Oc{{\mathcal O}}
\def\cV{{\mathcal V}}
\def\cZ{{\mathcal Z}}

\DeclareMathOperator*{\argmin}{arg\,min}

\title[$C^1$ and $C^2$-convergence to weak K.A.M. solutions]{On the $C^1$ and $C^2$-convergence to weak K.A.M. solutions}

\author[M.-C. Arnaud] {Marie-Claude Arnaud}
\address{Avignon Universit\'e\\
 Laboratoire de Math\'ematiques d'Avignon 
(EA 2151)\\ 
F-84018  Avignon, France}
\email{Marie-Claude.Arnaud@univ-avignon.fr}
\thanks{$\ddag$   member of the {\sl Institut universitaire de France.}}

\author[X. Su] {Xifeng Su}
\address{ School of Mathematical Sciences\\
Beijing Normal University\\
No. 19, XinJieKouWai St.,HaiDian District\\
 Beijing 100875, P. R. China}
\email{xfsu@bnu.edu.cn, billy3492@gmail.com}

\begin{document}

\maketitle
\vspace{0.1in}

\begin{abstract}

We introduce a notion of upper Green regular solutions to the Lax-Oleinik semi-group that is defined on the set of $C^0$ functions of a closed manifold via a Tonelli Lagrangian. Then we prove some weak $C^2$ convergence results to such a solution for a large class of approximated solutions as
\begin{enumerate}
\item the discounted solution (see \cite{DFIZ16});
\item the image of a $C^0$ function by the Lax-Oleinik semi-group;
\item the weak K.A.M. solutions for perturbed cohomology class.
\end{enumerate}
This kind of convergence implies the convergence in  measure of the second derivatives.

Moreover, we provide an example that is not upper Green regular and to which we have $C^1$ convergence but not convergence in measure of the second derivatives.

\end{abstract}

\section{Introduction}
This article focuses on some weak solutions of the stationary Hamilton-Jacobi equation $H(\cdot, du(\cdot))=c$ on some closed manifold $M^{(d)}$. Classical solutions of this equation are generating functions of  Lagrangian submanifolds that are invariant by  the Hamiltonian flow, but it often happens that such classical solutions don't exist. 

The viscosity solutions were then introduced  by P.-L. Lions and M.G. Crandall (see \cite{CraLi1983}) and provide generalized solutions under very weak hypotheses for $H$. In 1997 and in a convex setting, A.~Fathi proved his weak K.A.M. theorem (see \cite{Fathi1997}) that provides weak K.A.M. solutions and also proved (see \cite{Fathi2008}) that these solutions coincide with the viscosity solutions. The weak K.A.M. solutions are   fixed points of the so-called Lax-Oleinik semi-group and Fathi proved in \cite{Fathi1998} the convergence of the Lax-Oleinik semi-group to weak K.A.M. solutions in $C^0$-topology. 

Here we consider various problems of $C^1$ and $C^2$ convergence, that correspond to a convergence of graphs of discontinuous Lagrangian submanifolds of $T^*M$. Before our results, only results concerning the $C^0$-convergence were known. 



We study the problem of the $C^1$ or $C^2$ convergence of approximated solutions for the Lax-Oleinik semi-group defined on a closed manifold $M$. More precisely, we will consider the  following three problems:
\begin{enumerate}
\item \label{Pt2} the dependence of the weak K.A.M. solution on the cohomology class;
\item the convergence of the so-called discounted solution (see \cite{DFIZ16});
\item \label{Pt1} the convergence of the Lax-Oleinik semi-group to a weak K.A.M. solution.
\end{enumerate}

The problem of  $C^1$ convergence for Point (\ref{Pt1}) was partially solved in \cite{Arnaud2005}. 
The Dynamics that we will consider are Hamiltonian or conformally Hamiltonian on $T^*M$ and are all  convex in the fiber, which means the following.
\begin{definition}
A $C^2$ function $H:(q, p)\in T^*M \mapsto H(q, p)\in \R$ is $C^2$-convex in the fiber direction if for every $x\in T^*M$, the Hessian in the fiber direction $\frac{\partial^2 H}{\partial p^2}(x)$, denoted by $H_{p,p}(x)$ for short, 
 is positive definite as a quadratic form.\\
The $C^2$ function  $H$ is superlinear in the fiber direction if for any Riemannian metric on $M$,  for any $A>0$, there exists $B$ such that
$$\forall (q, p)\in T^*M, H(q, p)\geq A\| p\|+B.$$
A Tonelli Hamiltonian is a function that is superlinear and $C^2$-convex in the fiber direction.
\end{definition}

We will be interested in conformally Hamiltonian flows associated to a Tonelli Hamiltonian, defined by the following equations (see \cite{MaSo2016}) with $\lambda>0$. 
\begin{equation} \label{EhamD}
\frac{dq}{dt}=\frac{\partial H}{\partial p}(q, p)\quad{\rm and}\quad\frac{dp}{dt}=-\frac{\partial H}{\partial q}(q, p)-\lambda\ p.\end{equation} 
Observe that the case $\lambda=0$ is the Hamiltonian case.

\begin{definition}[Hausdorff distance]
Let $(X,d)$ be a metric space. For any non-empty compact subsets $K_1, K_2$ of $X$, the Hausdorff distance between $K_1$ and $K_2$ is defined by
\[
d_H(K_1,K_2) := \max\{ \rho(K_1, K_2), \rho(K_2, K_1) \}
\]
where $\rho(K_1, K_2) = \sup_{x\in K_1} d(x, K_2)$.
\end{definition}

\begin{nota}
\begin{itemize}
\item Choosing a Riemannian metric, we will denote by $d_H$ the associated Hausdorff distance in $T^*M$;
\item if $\Lambda:N\subset M\rightarrow T^*M$ is a section of $\pi:T^*M\rightarrow M$, its {\em graph} is denoted by
$$\cG(\Lambda)=\{~ (q, \Lambda(q))~:~ q\in N~\}\subset T^*M;$$
\item at every $x\in T^*M$, the vertical subspace at $x$ is $V(x)=\ker D\pi(x)$;
\item  if $A$ is a subset of a topological space, we denote its closure by $\bar A$.
\end{itemize}
\end{nota}

\begin{theorem}\label{TC1discounted}
Let $H:T^*M\rightarrow \R$ be a $C^2$ Tonelli Hamiltonian. Let $(u_\lambda)_{\lambda\in (0, 1]}$ be the solutions to the associated discounted problem (see \cite{DFIZ16} ) and let $u:M\rightarrow \R$ be their  limit $\displaystyle{\lim_{\lambda\rightarrow 0^+ } u_\lambda=u}$. Then 
$$\lim_{\lambda\rightarrow 0^+ }d_H(\overline{\cG(du_\lambda)}, \overline{\cG(du)})=0.$$
\end{theorem}
\begin{cor}\label{CC1discounted}
With the same hypotheses as in Theorem \ref{TC1discounted}, if $u$ is $C^1$, then $u_\lambda$ converges to $u$ for the uniform $C^1$ topology when $\lambda\rightarrow 0^+$.
\end{cor}
To give a similar statement in the case of varying cohomology classes, we introduce some notations.

\begin{nota}
For every $c$ in the linear space $H^1(M, \R)$, we choose in a continuous way a smooth closed 1-form $\eta_c$ with cohomology class $c$. When $M=\T^d$, we can identify $H^1(\T^d, \R)$ with the set of constant 1-forms.
\end{nota} 

\begin{theorem}\label{TC1cohom}
Let $H:T^*M\rightarrow \R$ be a $C^2$ Tonelli Hamiltonian. For every $c\in H^1(M, \R)$, we consider the modified Lax-Oleinik semi-group $(T^c_t)_{t\in \R_+}$ that corresponds to the closed $1$-form $\eta_c$, defined by
{
\[
T^c_t u (x) = \inf_{\gamma} \left\{  u(\gamma(-t)) + \int_{-t}^0 \big[ L(\gamma(s), \dot{\gamma}(s) ) - \langle \eta_c, \dot{\gamma}(s)\rangle +\alpha(c) \big] \right\}
\]
where the infimum is taken over all the absolutely continuous curves $\gamma: [-t, 0] \rightarrow M$ such that $\gamma(0)=x$\footnote{$\alpha(c)$ is Ma\~n\'e critical value for the cohomology class $c$, see \cite{Fathi2008}.}.
}
Assume that  $(u_c)_{c\in D}$ is a family of  fixed points of $(T^c_{ t })$ that uniformly converge to $u$ when $c$ tends to $0$.\\
 Then, $$\lim_{c\rightarrow 0}d_H(\overline{\cG(\eta_c+du_c)}, \overline{\cG(\eta_0+du)})=0.$$
\end{theorem}

\begin{cor}\label{CC1cohom}
With the same hypotheses as in Theorem \ref{TC1cohom}, if $u$ is $C^1$, then $u_c$ converges to $u$ for the uniform $C^1$ topology when $c\rightarrow 0$.
\end{cor}
We will now focus on the case of $C^2$ topology when $M=\T^d$ 
and the considered limit solution $u$ satisfies some regularity assumption that we will detail.

For  Dynamics that are defined with a Tonelli Hamiltonian, the pieces of orbit with no conjugate points play a special role; for example, in a Lagrangian setting, they correspond to locally minimizing orbits.  
\begin{definition}
Let $(\varphi_t)_{t\in \R}$ be a   flow on $\T^d\times\R^d$. \\
\begin{itemize}
\item a piece of orbit $(\varphi_t(x))_{t\in I}$ with interval $I\subset \R$  has {\em no conjugate points} if
$$\forall t\not=s\in I, \left( D\varphi_{t-s}V(\varphi_s(x))\right)\cap V(\varphi_t(x))=\{ 0\};$$

\item for such a piece of orbit, for every $s, t\in I$, we define
$$G_{t-s}(\varphi_t(x))=D\varphi_{t-s}V(\varphi_s(x)).$$
\end{itemize}
\end{definition}
For such a piece of orbit with no conjugate points, observe that all the Lagrangian subspaces $G_{t-s}(\varphi_t(x))$ with $t\not=s$ are transverse to the vertical $V(\varphi_t(x))$ and then are graphs of some symmetric matrix  in the usual coordinates.

\begin{notas}
\begin{itemize}
\item  We denote the set of symmetric matrices with size $n$ by $\cS_n$.
\item Let $G\subset T_x (\T^d\times\R^d)$ be a Lagrangian subspace that is transverse to the vertical subspace. Its {\em height} $\cH(G){\in \cS_d}$ is the symmetric matrix such that 
$$G=\{ (\delta \theta,\cH(G)\delta\theta)~;~ \delta\theta\in \R^d\}.$$
 In fact, we will identify $\cH(G)$ with a  quadratic form.
\end{itemize}
\end{notas}
The set of symmetric matrices is endowed with a natural order, the one of the corresponding quadratic forms. The following proposition is  proved in  \cite{Arnaud08} for the Hamiltonian case and we will prove in Section~\ref{sGreen} that it is also true for conformal Hamiltonian flows.
\begin{prop}\label{PGreen}
If $(\varphi_t)_{t\in \R}$ is  a conformal Hamiltonian flow on $\T^d\times\R^d$ that is associated to a $C^2$-convex in the fiber Hamiltonian and if $(\varphi_t(x))_{t\in I}$ is a piece of orbit with no conjugate points, then 
\begin{itemize}
\item if $t\in I$, the map $s\in  (-\infty, t)\cap I\mapsto \cH(G_{t-s}(\varphi_t(x)))$ is increasing and the map $s\in  (t, +\infty)\cap I\mapsto \cH(G_{t-s}(\varphi_t(x)))$ is {increasing};
\item for every $s_1\in  I\cap (-\infty, t)$ and $s_2\in (t, +\infty)\cap I$, then $\cH(G_{t-s_1}(\varphi_t(x)))>\cH(G_{t-s_2}(\varphi_t(x)))$;
\item when $\inf I=-\infty$, the limit $\displaystyle{G_+(\varphi_t(x))=\lim_{s\rightarrow -\infty} G_{t-s}(\varphi_t(x))}$ exists and  when $\sup I=+\infty$, the limit $\displaystyle{G_-(\varphi_t(x))=\lim_{s\rightarrow +\infty} G_{t-s}(\varphi_t(x))}$ exists;
\item when $I=\R$, we have $\cH(G_-)\leq\cH(G_+)$.
\end{itemize}
\end{prop}
$G_-$ and $G_+$ are then called {\em Green bundles}.
\medskip

As said before, we will consider some special weak K.A.M. solutions $u:\T^d\rightarrow \R$ of some Hamiltonians. These solutions are always semi-concave\footnote{See Section~\ref{sdiscountedE} for the definition.}, and then
\begin{itemize}
\item they are Lipschitz and Lebesgue almost everywhere differentiable by Rademacher Theorem (see \cite{EvGa2015});

\item if $\theta_k\in \mathbb{T}^d$ converges to $\theta$ and if $p_k\in D^+u(\theta_k)$\footnote{ See Section~\ref{sdiscountedE}  for the notation.}  converges to a vector $p\in\mathbb{R}^d$, then $p\in D^+u(\theta)$ (see \cite{CannarsaS04})\footnote{ Here $D^+u(x)$ denotes the set of super-differentials of $u$ at $x$, see Section~\ref{sdiscountedE} for the definition.};

\item by Alexandrov Theorem (see \cite{NiPer2006}), they admit a second derivative $D^2u$ at Lebesgue almost every $\theta\in\T^d$. 
\end{itemize} 
It can be proved (see \cite{Fathi2008}) that at every point $\theta$ where the weak K.A.M. solution $u$ is differentiable, the negative orbit $(\varphi_t(\theta, du(\theta)))_{t\in\R_-}$ has no conjugate points and thus the Green bundle $G_+(\theta, du(\theta))$ exists.
\begin{definition}
A weak K.A.M. solution $u$ is {\em upper Green regular} if 
at Lebesgue almost every $\theta\in \T^d$, we have
$$\cH(G_+(\theta, du(\theta)))=D^2u(\theta).$$
A weak K.A.M. solution $u$ is {\em lower Green regular} if 
at Lebesgue almost every $\theta\in \T^d$, we have
$$\cH(G_-(\theta, du(\theta)))=D^2u(\theta).$$
\end{definition}
We will prove in Section \ref{Sexamples} that the following examples of restricted Dynamics to invariant $C^1$ Lagrangian graphs correspond to a $C^1$,  upper and lower Green regular weak K.A.M. solution
\begin{itemize}
\item the restricted Dynamics is Lipschitz conjugated to the one of a rotation flow;
\item the restricted Dynamics is Kupka-Smale;
\item the degree of freedom is $d=2$.
\end{itemize}
In particular, the K.A.M. tori  are graphs of derivatives of weak  K.A.M. solutions that are upper and lower Green regular. Hence we can apply our results of convergence to the K.A.M. tori case.

 We will now estimate
a kind of  $C^2$ distance between any $C^1$ and upper  (resp. lower) Green regular weak K.A.M. solution and its approximated solutions. The quantity that we will estimate is described below.

\begin{nota}
\begin{itemize}
\item We denote by {\rm Leb} the Lebesgue measure on $\T^d$.

\item If $S$ is a symmetric matrix on $\mathbb{R}^d$, its norm is defined by $$\displaystyle{\|S\|=\sup_{  v\in\mathbb{R}^d, \|v\|=1} |S(v, v)|}$$ {where $\|\cdot\|$ is the standard Euclidean norm we take on $\mathbb{R}^d$.}

\item Let $u, v: \mathbb{T}^d\rightarrow \R$ be two semi-concave functions. Then they admit a second derivative Lebesgue almost everywhere and we can define
$$d_{2,1}(u, v)=\int_{\T^d}\| D^2u(\theta)-D^2v(\theta)\|d{\rm Leb}(\theta).$$
\end{itemize}
\end{nota}
\begin{theorem} \label{TC2discounted}
Let $H:\T^d\times \R^d\rightarrow \R$ be a $C^2$ Tonelli Hamiltonian. Let $(u_\lambda)_{\lambda\in (0, 1]}$ be the solutions of the associated discounted problem  and let $u:\T^d\rightarrow \R$ be their  limit, i.e. $\displaystyle{\lim_{\lambda\rightarrow 0^+ } u_\lambda=u}$. Then, if $u$ is $C^1$ and upper Green regular, $u$ is a weak K.A.M. solution that satisfies
$$\lim_{\lambda\rightarrow 0^+ }d_{2,1}(u_\lambda,u)=0.$$
\end{theorem}
\begin{theorem}\label{TC2convLO}
Let $H:\T^d\times \R^d\rightarrow \R$ be a $C^2$ Tonelli Hamiltonian with associated Lax-Oleinik semi-group $(T_t)_{t\geq 0}$. Let $u_0\in C^0(\T^d, \R)$ and let us use the notation $u_t=T_t u_0$ and $\displaystyle{u=\lim_{t\rightarrow +\infty} u_t}$ \footnote{{The existence of the limit is due to weak K.A.M. theorem, see \cite{Fathi2008}.} }. 
Then, if $u$ is $C^1$ and upper Green regular,   $u$ is a weak K.A.M. solution that satisfies
$$\lim_{t\rightarrow +\infty}d_{2,1}( u_t,u)=0.$$
\end{theorem}
\begin{theorem}\label{TC2cohomclass}
Let $H:\T^d \times \R^d\rightarrow \R$ be a $C^2$ Tonelli Hamiltonian. For every $c\in\R^d$, we consider the modified Lax-Oleinik semi-group $(T^c_t)_{t\in \R_+}$ that corresponds to the cohomology class $c$ . Assume that  $(u_c)_{c\in D}$ is a family of fixed points of $(T^c_{t})$ that uniformly converge to $u$ when $c$ tends to $0$.\\
 Then, if $u$ is $C^1$ and upper Green regular, $u$ is a weak K.A.M. solution that satisfies
$$\lim_{c\rightarrow 0}d_{2,1}( u_c,u)=0.$$
\end{theorem}
In \cite{Fathi2008}, the symmetrical Lagrangian $\widetilde L(q,v)=L(q, -v)$ is introduced and the symmetrical Lax-Oleinik semi-group is defined. More precisely, if we just here adopt the notation $(T_t^L)_{t>0}$ for the Lax-Oleinik semi-group for $L$ and $(T_t^{L,\lambda})_{t>0}$ for the discounted semi-group for $L$, we define
\begin{itemize}
\item the symmetrical Lax-Oleinik semi-group $(\tilde T^L_t)_{t>0}$ is defined by $\tilde T^L_tu=-T^{\tilde L}_t(-u);$
\item the symmetrical discounted semi-group $(\tilde T^{L, \lambda}_t)_{t>0}$ is defined by $\tilde T^{L, \lambda}_tu=-T^{\tilde L, \lambda}(-u)$.
\end{itemize}

Using its definition, we deduce easily for $C^1$ and lower Green regular solutions $u$ to the symmetrical semi-group the $d_{2,1}$ convergence of
\begin{itemize}
\item the symmetrical discounted solutions;
\item the image  of an initial condition by the symmetrical Lax-Oleinik semi-group; 
\item the symmetrical solutions depending on the cohomology class.
\end{itemize}

\begin{remks}
\begin{itemize}
\item For Theorems \ref{TC2discounted},  \ref{TC2convLO} and \ref{TC2cohomclass}, the fact that $M=\T^d$ is not fundamental. But to give some  correct statements on any closed manifold, we would need to choose a ``horizontal'' subspace at any point by using a connection. We preferred to avoid this, but a similar proof (in charts) could be given for any closed manifold.
\item Observe that this kind of convergence implies the convergence to $0$ in (Lebesgue) measure of the $C^2$-distances to the limit, {for instance, in the case of Theorem~\ref{TC2cohomclass},}   i.e.
$$\forall \varepsilon>0, \lim_{c\rightarrow 0} Leb\left(\{ \theta\in \T^d;\| D^2u(\theta)-D^2u_c(\theta)\|\geq \varepsilon\}\right)=0.$$
\item We will see in Subsection \ref{ssnotunif} by providing some example  that we cannot improve this convergence in a uniform one for the $C^2$-distance $d_{2,1}$. 
\end{itemize}
\end{remks}
Moreover, we will build in Subsection \ref{ssnotup} an example on a weak K.A.M. solution that is not upper Green regular nor lower Green regular and we will prove for this example that the conclusion of Theorem \ref{TC2convLO} is not valid. Note that for this {example}, we will not work on a torus $\T^d$.

We end this introduction by {asking} some question.

\begin{ques}
Does there exist an example of  {which} a weak K.A.M. solution is $C^1$, upper Green regular but {not} lower Green regular?
\end{ques}

\subsection{Notations}\label{firstnota} 
 As said before, $\pi: T^*M\rightarrow M$ is the {\em canonical projection} and the {\em vertical subspace} at $x\in T^*M$ is $V(x)=\ker D\pi(x)$.\\
We recall that if $q=(q_i)_{1\leq i\leq d}$ are coordinates in $M$, we define dual coordinates $(p_i)_{1\leq i\leq d}$ as follows: if $\eta$ is an element of $T_q^*M$, it can be written in the basis $(dq_1, \dots , dq_d)$ as $\displaystyle{\eta=\sum_{i=1}^d p_idq_i}$, and then the coordinates of $\eta$ are $(q_1, \dots, q_d, p_1, \dots, p_d)$.\\
The usual {\em symplectic form} $\omega$ on $T^*M$ is chosen in such a way that all these coordinates are symplectic. In other words, we have in dual coordinates
$$\omega=\sum_{i=1}^d dq_i\wedge dp_i.$$
$M$ and then  $T^*M$ are endowed with a Riemannian metric and we denote by $B(x, r)$ the {\em open ball} with center $x$ and radius $r$.

\section{Basic facts about discounted equation}\label{sdiscountedE}

\subsection{Semi-concave functions}
\begin{definition}
A function $u:U\rightarrow \R$ defined on an open subset $U$ of $\R^d$ is {\em semi-concave} if there exists some constant $K\in \R$ such that
$$\forall x\in U, \exists p\in \R^d, \forall y\in U, u(y)\leq u(x)+p(y-x)+K\| y-x\|^2.$$
We also say that $u$ is {\em $K$-semi-concave}.\\
Then $p$ is a {\em super-differential} of $u$ at $x$ and we denote by $D^+u(x)$ the set of super-differentials of $u$ at $x$.

If $M$ is a closed manifold, we fix a finite atlas $\cA=\{ (\phi, V)\}$. A function $u: M\rightarrow \R$ is said to be $K$-semi-concave if every $u\circ \phi^{-1}$ is $K$-semi-concave and $p\in D^+u(x)$ means that $p\circ D\phi(x)\in D^+(u\circ \phi^{-1})$(x).
\end{definition}
A good reference for semi-concave functions is  \cite{CannarsaS04}.  We recall that a semi-concave function is always locally the sum of a concave function and a smooth function. We recalled in the introduction the following properties of the semi-concave functions.
\begin{itemize}
\item They are Lipschitz and Lebesgue almost everywhere differentiable by Rademacher Theorem (see \cite{EvGa2015});

\item if $q_k\in M$ converges to $q$ and if $p_k\in D^+u(q_k)$  converges to some $p\in T^*M$, then $p\in D^+u(q)$ (see \cite{CannarsaS04});

\item by Alexandrov Theorem (see \cite{NiPer2006}), they admit a second derivative $D^2u$ at Lebesgue almost every $q\in M$. 
\end{itemize}

Let $H:T^*M \rightarrow \mathbb{R}$ be a $C^2$ Tonelli Hamiltonian and $L:TM\rightarrow \mathbb{R}$ be its associated Lagrangian via the Legendre transformation. $\alpha(H)$ is the Ma\~n\'e critical value of $H$. 
\begin{definition}
A function $u\in C(M, \mathbb{R})$ is called a weak K.A.M. solution of negative type of Hamilton-Jacobi equation
\begin{equation}\label{HJequation}
H(x, du(x)  ) = \alpha(H)
\end{equation}
 if
\begin{itemize}
\item [(i)] for each continuous piecewise $C^1$ curve $\gamma: [t_1, t_2] \rightarrow M$ with $t_1<t_2$, we have
\[
u(\gamma(t_2)) - u(\gamma(t_1))  \leq \int_{t_1}^{t_2} \big[L(\gamma(s), \dot{\gamma}(s)) + \alpha(H) \big] \  ds;
\]

\item [(ii)] for any $x\in M$, there exists a $C^1$ curve $\gamma: (-\infty, 0] \rightarrow M$ with $\gamma(0) = x$ such that for any $t\geq 0$, we have
\[
u(x) - u(\gamma(-t)) = \int_{-t}^0\big[L(\gamma(s), \dot{\gamma}(s)) + \alpha(H) \big] \  ds.
\]
\end{itemize}
\end{definition}

A discounted version of \eqref{HJequation} is the equation
\begin{equation}\label{discounted HJequation}
\lambda u(x) + H(x, du(x)  ) = \alpha(H)
\end{equation}where $\lambda>0$. 
Note that the viscosity solution of \eqref{discounted HJequation} is unique and denoted by $u_\lambda$.  We call $u_\lambda$ the discounted solutions of \eqref{HJequation}  and  it can be represented by the following formula
\[
u_\lambda(x)  = \inf_{\gamma} \int_{-\infty}^0 e^{\lambda s} \big[ L(\gamma(s), \dot{\gamma}(s)) + \alpha(H)\big] \ ds, \qquad \forall ~x\in M
\]
where the infimum is taken over all absolutely continuous curves $\gamma: (-\infty, 0] \rightarrow M$ with $\gamma(0) =x$.

\subsection{Discounted Dynamics}
We assume that $H:T^*M\rightarrow \R$ is a Tonelli Hamiltonian. Let $L:TM\rightarrow \R$ be the Lagrangian associated to $H$. 

We denote by $(\varphi_t^\lambda)$  the flow that solves Equation (\ref{EhamD}) that we recall:
$$(\ref{EhamD})\qquad
\frac{dq}{dt}=\frac{\partial H}{\partial p}(q, p)\quad{\rm and}\quad\frac{dp}{dt}=-\frac{\partial H}{\partial q}(q, p)-\lambda\ p.$$
Recall that the Legendre map $\cL: T^*M\rightarrow TM$ is a diffeomorphism that is defined by
$$\cL(q, p)=(q, \frac{\partial H}{\partial p}(q, p))$$
and we have
$$\cL^{-1}(q, v)=(q, \frac{\partial L}{\partial v}(q, v)).$$
Then the flow $(f^\lambda_t)=(\cL\circ \varphi_t^\lambda\circ \cL^{-1})$ solves the discounted Euler-Lagrange equation
\begin{equation}\label{ELAgD}\frac{d}{dt}\left( \frac{\partial L}{\partial v}(\gamma, \dot\gamma)\right)-\frac{\partial L}{\partial q}(\gamma, \dot\gamma) + \lambda\frac{\partial L}{\partial v}(\gamma, \dot\gamma)=0.\end{equation}
For any $\lambda \in \R$ and $t>0$, we define the following action on $M\times M$
\begin{equation}\label{Eaction}
a_t^\lambda(q_0, q_1)=\inf_{\gamma}\int_{-t}^0e^{\lambda s} \left[ L(\gamma(s), \dot\gamma(s)) +\alpha(H) \right] ds
\end{equation}
where the infimum is taken on all the absolutely continuous curves $\gamma:[-t, 0]\rightarrow M$ such that $\gamma(-t)=q_0$ and $\gamma(0)=q_1$.\\
Then the infimum in Equality (\ref{Eaction}) is a minimum and every $\gamma$ where this minimum is reached corresponds to a solution of the $\lambda$-discounted  Euler-Lagrange equation, i.e. satisfies
 $$(\ref{ELAgD})\quad \frac{d}{dt}\left( \frac{\partial L}{\partial v}(\gamma, \dot\gamma)\right)-\frac{\partial L}{\partial q}(\gamma, \dot\gamma) + \lambda\frac{\partial L}{\partial v}(\gamma, \dot\gamma)=0.$$

Then $\gamma$ is a minimizing curve and the corresponding orbits for the Euler-Lagrange and Hamiltonian flows are said to be {\em minimizing}.
 \begin{prop}
 Any minimizing orbit has no conjugate points.
 \end{prop}
\begin{proof}
Observe that if we define $\tilde L(q, v, t)=e^{\lambda t}L(q, v)$, Equation (\ref{ELAgD}) is nothing else than the classical Euler-Lagrange equation for the time-dependent Lagrangian $\tilde L$. For such an equation, it is well-known that along any minimizing orbit, there are no conjugate points. Using Legendre map, there are also no conjugate points for the corresponding Hamiltonian orbit.
 \end{proof}

\subsection{Discounted  Lax-Oleinik semi-groups}\label{ssdiscLO}
Using methods similar to the ones used in \cite{Bernard08}, it can be proved that  
\begin{itemize}
\item every function $a_t^\lambda$ is semi-concave;

\item for every minimizing curve $\gamma$ in (\ref{Eaction}), $-e^{-\lambda t}\frac{\partial L}{\partial v}(\gamma(-t),\dot\gamma(-t)) $ is a super-differential of $a^\lambda_t(\cdot, \gamma(0))$ at $\gamma(-t)$ and $\frac{\partial L}{\partial v}(\gamma(0),\dot\gamma(0)) $ is a super-differential of $a^\lambda_t(\gamma(-t), \cdot )$ at $\gamma(0)$;  

\item       at $(q_0, q_1)$,      $a_t^\lambda$ admits a derivative with respect to the first variable if and only if it admits a derivative with respect to the second variable if and only if there  is only one minimizing curve $\gamma$ between $(-t, q_0)$ and $(0, q_1)$. Then in this case, we have   
$$\frac{\partial    a_t^\lambda}{\partial q_0}(q_0, q_1)=  -e^{-\lambda t}\frac{\partial L}{\partial v}(\gamma(-t),\dot\gamma(-t)) \quad{\rm and}\quad   \frac{\partial    a_t^\lambda}{\partial q_1}(q_0, q_1)=\frac{\partial L}{\partial v}(\gamma(0),\dot\gamma(0)).$$
\end{itemize}
 
The discounted Lax-Oleinik semi-group $(T^\lambda_t)_{t>0}$ is defined on the set of continuous functions $u:M\rightarrow \R$ by
\begin{equation}\label{EminDLO}T_t^\lambda u(q)=\inf_{\gamma} \left( e^{-\lambda t}u(\gamma(-t))+\int_{-t}^0e^{\lambda s} \big[ L(\gamma(s), \dot\gamma(s))+\alpha(H) \big]\  ds\right),\end{equation}
where the infimum is taken on all the absolutely continuous curves $\gamma:[-t, 0]\rightarrow M$ such that $\gamma(0)=q$. Then $T^\lambda_tu$ is semi-concave for any $t>0$.  This infimum is always a minimum and when $\gamma:[-t, 0]\rightarrow M$ is minimizing in Equation (\ref{EminDLO}), then $\gamma$ is a solution for (\ref{ELAgD}), $\frac{\partial L}{\partial v}(\gamma(-t), \dot\gamma(-t))$ is a sub-differential of $u$ at $\gamma(-t)$ and 
$\frac{\partial L}{\partial v}(\gamma(0), \dot\gamma(0))$ is a super-differential of $T_t^\lambda u$ at $\gamma(0)$.

Moreover, when $u$ is semi-concave, then $u$ is differentiable at $\gamma(-t)$. In this case, we have
\begin{equation}
\label{Emingraph} \varphi_{t}^{\lambda}(\gamma(-t), du (\gamma(-t)))= (q, \frac{\partial L}{\partial v}(\gamma(0), \dot\gamma(0)))\in \overline{\cG(d {T_t^\lambda} u)}.
\end{equation}

As every $T^\lambda_tu$ is semi-concave, it is Lipschitz and  differentiable on a subset $\cD\subset M$ that has full Lebesgue measure. Then if $q_{0}\in \cD$, there is only one minimizing curve in Equality (\ref{EminDLO}), that is given by $\gamma (s)=\pi\circ \varphi_{s} ^{\lambda}(q_{0}, dT^\lambda_t u(q_0))$ for any $s\in[-t, 0]$.

Observe that $\big(q_s, p_s)_{s\in [-t, 0]}=(\varphi_{s}^{\lambda}(q_{0}, dT^\lambda_tu(q_0))\big)_{s\in [-t, 0]}$ is a piece of orbit for the discounted Hamiltonian flow that joins a point of $\cG(d u)$ to a point of $\cG(dT_t^\lambda u)$ and then
\begin{itemize}
\item for every $s\in [-t, 0]$, $T_{t+s}^\lambda u$ is differentiable at $q_s=\gamma(s)$;
\item for every $s\in [-t, 0]$, $$(q_s, p_s)\in \cG(dT_{s+t}^\lambda u)  \subset \varphi^\lambda_{t+s}(\cG(d u)).$$
\end{itemize} 
Observe that this implies that
\begin{equation}\label{Einclusgraph} \cG(dT_t^\lambda u)\subset \varphi^\lambda_t(\cG(d  u)).\end{equation}

\subsection{A priori compactness results}
\begin{prop}[A priori compactness]\label{Papcomp}
Let $L$ be a Tonelli Lagrangian, $\lambda>0$ and $t>0$. There exist a neighborhood $\mathscr{N}$ of $(L, \lambda)$ in the compact-open $C^2$ topology and a compact set $\mathscr{K}_t\subset T M$ such that if $(L', \lambda') \in \mathscr{N}$ with $L'$ Tonelli and $\lambda'>0$ and if $\gamma:[-t, 0]\rightarrow M$ is a minimizing orbit for $(L', \lambda')$, then 
\[
(\gamma(s), \dot{\gamma}(s)) \in \mathscr{K}_t \qquad \forall s\in [-t,0].
\]
\end{prop}

\begin{proof} We fix $\varepsilon>0$. \\
Step 1. Fix a Riemannian metric $g$ on $M$ and $t>0$. Let $\gamma_{q_0, q_1}: [-t, 0] \rightarrow M$ be a geodesic for the metric $g$ joining $q_0$ and $q_1$. We have  
\[
\|\dot{\gamma}_{q_0,q_1}(s) \|_{\gamma_{q_0,q_1}(s)} = \frac{d(q_0, q_1)}{t} \qquad s\in [-t, 0].
\]
Consequently, the compact set 
\[
\mathscr{K} =\left \{ (q, v ) \in TM : \|v \| \leq \frac{\text{diam}(M)}{t} \right\}
\]
contains all the points $(\gamma_{q_0, q_1} (s), \dot{\gamma}_{q_0, q_1} (s))$ for $s\in [-t, 0]$.

Let $M_L := \max \{ \max\limits_{\mathscr{K} } L(q, v), 0\}$.
\begin{equation}
\begin{split}\label{Ecomp}
a_t^{L,\lambda} (q_0, q_1) &= \inf_\eta \int_{-t}^0 e^{\lambda s } L(\eta(s), \dot{\eta}(s) )  ds \\
&\leq \int_{-t}^0 e^{\lambda s } L(\gamma_{q_0,q_1}(s), \dot{\gamma}_{q_0,q_1}(s) )  ds\\
&\leq \int_{-t}^0 e^{\lambda s } M_L ds \leq M_L t
\end{split}
\end{equation}{where the infimum is taken over all the absolutely continuous curves $\eta: [-t, 0]\rightarrow M$ such that $\eta(-t) = q_0$ and $\eta(0) = q_1$.}

By the superlinearity of $L$, there exists $R > 2 \frac{\text{diam}(M)} {t} $ such that if $\|v\| \geq R$ we have 
\[L(q, v) \geq (M_L + \varepsilon) (1 + e^{(\lambda + \varepsilon) t})+\varepsilon.\]

We introduce the notation
\[
\mathscr{K}_1 = \{ (q,v)\in TM: \|v\| \leq R \}.
\]

Step 2. We consider any Tonelli Lagrangian $L'$ satisfying  $\| L' - L\|_{C^2, \mathscr{K}_1} \leq \varepsilon$ and any $\lambda'>0$ satisfying $|\lambda' - \lambda| \leq \varepsilon$.

We deduce from the definition of  $R$ and $\mathscr{K}_1$ and the inequality that $\| L' - L\|_{C^2, \mathscr{K}_1} \leq \varepsilon$ that if $\| v\|=R$, we have
$$ L'(q, v) \geq (M_L + \varepsilon) (1 + e^{(\lambda + \varepsilon) t}).$$

For every $(q, v)\in TM$ with $\|v\| {>} R$, let $w = \frac{R}{\|v\|} v$. By the convexity of $L'$, we have
\[
L'(q, w) \leq  ( 1- \frac{R}{\|v\|} ) L'(q,0) + \frac{R}{\|v\|} L'(q, v).
\]
So
\begin{equation*}
\begin{split}
L'(q, v) &\geq \frac{\|v\|}{R} L'(q, w) - (\frac{\|v\|}{R} -1) L'(q, 0)\\
&\geq   \frac{\|v\|}{R} (M_L + \varepsilon)(1+e^{(\lambda+ \varepsilon) t}) -  (\frac{\|v\|}{R} -1) (M_L + \varepsilon)\\
&=  \frac{\|v\|}{R} (M_L + \varepsilon) e^{(\lambda + \varepsilon) t} + M_L + \varepsilon >   (M_L + \varepsilon) e^{(\lambda + \varepsilon) t}.
\end{split}
\end{equation*}

We have then proven that
\begin{equation}\label{EK1} \forall (q, v)\notin \mathscr{K}_1, L'(q, v) > (M_L + \varepsilon) e^{(\lambda + \varepsilon) t}.\end{equation} 
Because  $\| L' - L\|_{C^2, \mathscr{K}_1} \leq \varepsilon$, we have $M_{L'}\leq M_L+\varepsilon$ and we deduce from (\ref{Ecomp}) that
$a_t^{L', \lambda'} (q_0, q_1)  \leq (M_L + \varepsilon) t$. That is,
if $\gamma$ is minimizing for $(L',\lambda')$ between $-t$ and $0$, we have 
\[
\int_{-t}^0 e^{\lambda' s } L'(\gamma(s),\dot{\gamma}(s)) ds \leq (M_L +\varepsilon) t .
\]
Hence, there exists $s_0\in[-t, 0]$ such that 
\[
L'(\gamma(s_0), \dot{\gamma}(s_0)) \leq (M_L + \varepsilon) e^{-\lambda' s_0} \leq (M_L + \varepsilon) e^{(\lambda + \varepsilon) t}.
\]
We deduce from Equation (\ref{EK1}) that $(\gamma(s_0), \dot{\gamma}(s_0)) \in \mathscr{K}_1$.

Hence, if $\gamma$ is minimizing for $(L',\lambda')$ between $-t$ and $0$, we have
\[
\forall s\in [-t, 0], \qquad(\gamma(s), \dot{\gamma}(s)) \in \phi_{L', \lambda'}^{[-t, t]} (\mathscr{K}_1).
\]
To conclude, note that  the set
\[
\bigcup_{\substack{\|L'-L\|_{C^2, \mathscr{K}_{1}} \leq \varepsilon\\ |\lambda'-\lambda| \leq \varepsilon}} \phi_{L', \lambda'}^{[-t, t]} (\mathscr{K}_1)
\]
is relatively compact in $TM$ because of the continuous dependence of the solutions of a differential equation from the parameters (see e.g. \cite{HubWe1995}).
\end{proof}

Using Legendre duality, we deduce a similar statement for Tonelli Hamiltonians.
\begin{cor}\label{Capcomp}
Let $H$ be a Tonelli Hamiltonian, $\lambda>0$ and $t>0$. There exist a neighborhood $\mathscr{N}$ of $(H, \lambda)$ in the compact-open $C^2$ topology and a compact set $\mathscr{K}_t\subset T^* M$ such that if $(H', \lambda') \in \mathscr{N}$ with $H'$ Tonelli and $\lambda'>0$ and if $(\varphi_s^H(x))_{s\in[-t, 0]}$ is a minimizing piece of orbit for $(H', \lambda')$, then 
\[
\varphi_s^H(x) \in \mathscr{K}_t \qquad \forall s\in [-t,0].
\]
\end{cor}

\section{Green bundles}\label{sGreen}
Green bundles will be the main ingredient to prove the results of $C^2$ convergence. Here we state some of their properties.
\subsection{Proof of Proposition \ref{PGreen}}
The first goal of this section is to prove Proposition \ref{PGreen}. The proof is very similar to the one given in \cite{Arnaud08} for Tonelli Hamiltonian flows. With the notations of Proposition \ref{PGreen}, we use $I_-=(-\infty, t)\cap I$ and $I_+= (t, +\infty)\cap I$.

Because there are no conjugate points on $I$, we have for every $s\not=s'$ in $I$
$$D\varphi_{s'-s}V(\varphi_s(x))\cap V(\varphi_{s'}(x))\not=\{ 0\}$$
and then  by taking their images by $D\varphi_{t-s'}(x)$, $$\cH(G_{t-s}(\varphi_t(x)))-\cH(G_{t-s'}(\varphi_t(x)))$$ is always a non-degenerate symmetric matrix. As this continuously depends on $s$, $s'$, we deduce that its signature is constant on each connected set
$$D_-=\{ (s, s')\in I_-^2, s<s'\}; D_+=\{ (s, s')\in I_+^2, s<s'\}; D_0=I_-\times I_+.$$

To determine these three signatures, we only consider the case where $|s|$ and $|s'|$ are small. We use the notation in usual coordinates for $D\varphi_s(y)$
$$M_s(y)=\begin{pmatrix}a_s(y)& b_s(y)\\ c_s(y)&d_s(y)
\end{pmatrix}.$$
Then we have $d_s(y)={\bf 1}+O(s)$ and $b_s(y)= O(s)$ and we deduce from linearized discounted equations that $\dot b_s(y) = H_{q, p}(y)b_s(y)+H_{p, p}(y)d_s(y) = H_{p, p}(y)+O(s)$ and then $b_s(y) = sH_{p, p}(y)+O(s)$, the $O(s)$ being uniform in $y$. This implies that 
$$\cH(G_s(\varphi_t(x)))=d_s(\varphi_{t-s}(x)).\left( b_s(\varphi_{t-s}(x)) \right)^{-1}=\frac{1}{s}(H_{p, p}(\varphi_t(x))^{-1}+O(s)).$$
We deduce for $s>0$ small enough that
$$\cH(G_{2s}(\varphi_t(x)))-\cH(G_s(\varphi_t(x)))= - \frac{1}{2s}(H_{p, p}(\varphi_t(x))^{-1}+O(s))  < 0;$$
$$\cH(G_{s}(\varphi_t(x)))-\cH(G_{-s}(\varphi_t(x)))=\frac{2}{s}(H_{p, p}(\varphi_t(x))^{-1}+O(s))>0;$$
$$\cH(G_{-s}(\varphi_t(x)))-\cH(G_{-2s}(\varphi_t(x)))=-\frac{1}{2s}(H_{p, p}(\varphi_t(x))^{-1}+O(s))<0.$$
This finishes the proof of Proposition \ref{PGreen}.
\subsection{Continuity of $ G_t^{c,\lambda}$ and semi-continuity of the two  Green bundles}
\begin{nota}\begin{itemize}
\item We consider a map $c\in C \mapsto H_c$ defined on some metric space $C$ that is continuous for the $C^2$ open-compact topology and such that every $H_c: T^*M\rightarrow \R$ is a $C^2$-convex in the fiber Hamiltonian. We will denote by $(\varphi_t^{c, \lambda})$ the flow associated to the $\lambda$-discounted equation for $H_c$.
\item then we use the notation $G_t^{c, \lambda}(x)=D\varphi_t^{c, \lambda}V(\varphi_{-t}^{c, \lambda}(x))$.
\end{itemize}
\end{nota}
Observe that the map 
$$g:(t, x, c, \lambda)\mapsto G_{-t}^{c, \lambda}(x)$$
is continuous. 

\begin{nota}
We then define $\cU$ as being the set of the
$(t, x, c, \lambda)\in \R\times T^*M\times C\times \R$ such that there is no conjugate points for $H^{c, \lambda}$ on the piece of orbit of $x$ between $x$ and $\varphi_t^{c, \lambda}(x)$. \\
Moreover $\cU_-=\cU\cap (\R_-\times T^*M\times C\times \R)$ and $\cU_+=\cU\cap (\R_+\times T^*M\times C\times \R)$. 
\end{nota}
Because $g$ is continuous, $\cU$ is open and the map $h=\cH\circ g:\cU_-\cup \cU_+\rightarrow \cS_d$ is continuous.
We deduce from Proposition~\ref{PGreen} that $h$ is increasing  in the first variable on $\cU_+$ (resp. $\cU_-$) and that if $(t, x, c, \lambda)\in \cU_-$ and $(s, x, c, \lambda)\in \cU_+$,  we have
\begin{equation}\label{EfiniGB} h({s}, x, c, \lambda)<h({t}, x, c, \lambda).\end{equation}
\begin{nota}
We are interested in infinite time interval, so we introduce
$$\cU_-^\infty=\{ (x,c, \lambda); \forall t\in \R_-, (t, x, c, \lambda)\in\cU_-\}$$
and 
$$\cU_+^\infty=\{ (x,c, \lambda); \forall t\in \R_+, (t, x, c, \lambda)\in\cU_+\}.$$
\end{nota}
We deduce from the continuity of $h$ that $\cU_-^\infty$ and $\cU_+^\infty$ are closed. Moreover, we can define
\begin{itemize}
\item for $(x,c, \lambda)\in \cU_-^\infty$ the Green bundle $\displaystyle{G_{+}^{c, \lambda}(x)=\lim_{t\rightarrow+\infty}G_{t}^{c, \lambda}(x)}$;
\item for $(x,c, \lambda)\in \cU_+^\infty$ the Green bundle $\displaystyle{G_{ -}^{c, \lambda}(x)=\lim_{t\rightarrow+\infty}G_{-t}^{c, \lambda}(x)}$.
\end{itemize}
Then we have
\begin{itemize}
\item $\displaystyle{ h_-(x, c, \lambda):=\cH(G_-^{c, \lambda}(x))=\lim_{t\rightarrow+\infty}h({t}, x, c, \lambda)}$;
\item $\displaystyle{h_+(x, c, \lambda):=\cH(G_+^{c, \lambda}(x))=\lim_{t\rightarrow+\infty}h({-t}, x, c, \lambda)}$.

\end{itemize}
Observe that because of Equation (\ref{EfiniGB}), we have
$$\forall (x, c, \lambda)\in \cU_-^\infty\cap \cU_+^\infty, h_-(x, c, \lambda)\leq h_+(x, c, \lambda).$$
We deduce from the fact that the considered functions are continuous and $t$-increasing the following proposition about semi-continuity.
\begin{prop}
Let us fix $(x, c_0, \lambda_0)\in \cU^\infty_{ +}$ and $\varepsilon>0$. Then there exist a neighborhood $\cN_-$ of $(x, c_0, \lambda_0)$ in $T^*M\times C\times \R$ and $T>0$ such that
\begin{itemize}
\item for every  $({t}, y, c, \lambda)\in \cU_{+}\cap (\R_{+}\times \cN_-)$ with $t\geq T$, we have $$h({ t}, y, c, \lambda)\geq h_-(x, c_0, \lambda_0)-\varepsilon{\bf 1};$$
\item   for every $(y, c, \lambda)\in \cN_-\cap\cU_{+}^\infty$, we have $$h_-(y, c, \lambda)\geq h_-(x, c_0, \lambda_0)-\varepsilon{\bf 1}.$$
\end{itemize}
\end{prop}
\begin{proof} The second point comes from the first point by taking the limit for $t\rightarrow+\infty$. Now we prove the first point.\\
Because $\displaystyle{\lim_{t\rightarrow +\infty}h({t}, x, c_0, \lambda_0)=h_-(x, c_0, \lambda_0)}$, there exists some $T>0$ such that
$$h({T}, x, c_0, \lambda_0)>h_-(x, c_0, \lambda_0)-\varepsilon{\bf 1}.$$
By continuity of $h$, there exists a neighborhood $\cN_-$ of $(x, c_0, \lambda_0)$ in $T^*M\times C\times \R$ such that
$$\forall (y, c, \lambda)\in \cN_-, ({T}, y, c, \lambda)\in \cU_{+}\Rightarrow  h({T}, y, c, \lambda)>h_-(x, c_0, \lambda_0)-\varepsilon{\bf 1}.$$
Because $h$ is increasing in $t$, we have
$$\forall t\geq T, \forall (y, c, \lambda)\in \cN_-, ({t},y, c, \lambda)\in \cU_{+}\Rightarrow h({t}, y, c, \lambda)>h_-(x, c_0, \lambda_0)-\varepsilon{\bf 1}.$$
\end{proof}
We have of course in a similar way a statement for the positive times.
\begin{prop}\label{PupsemicGreen}
Let us fix $(x, c_0, \lambda_0)\in \cU^\infty_{\color{blue} -}$ and $\varepsilon>0$. Then there exists a neighborhood $\cN_+$ of $(x, c_0, \lambda_0)$ in $T^*M\times C\times \R$ and $T>0$ such that
\begin{itemize}
\item for every $({-t}, y, c, \lambda)\in \cU_{-}\cap (\R_{-}\times \cN_+)$  with $t\geq T$, we have $$h({-t}, y, c, \lambda)\leq h_+(x, c_0, \lambda_0)+\varepsilon{\bf 1};$$
\item   for every $(y, c, \lambda)\in \cN_+\cap\cU_{-}^\infty$, we have $$h_+(y, c, \lambda)\leq h_{+}(x, c_0, \lambda_0)+\varepsilon{\bf 1}.$$
\end{itemize}
\end{prop}
\subsection{Comparison between Green bundles and second derivatives}
\begin{prop}\label{tLOGreen}
Let $u\in C^0(M, \R)$ and $t>0$.  Then for every point $q_0$ where $T^\lambda_tu$ is twice differentiable
\begin{itemize}
\item $(\varphi_s^\lambda(q_0, dT^\lambda_tu(q_0)))_{s\in [-t, 0]}$ has no conjugate points;
\item $D^2 T_t^{\lambda} u(q_0)\leq \cH(G^\lambda_t(q_0, dT^\lambda_tu(q_0)))$.
\end{itemize}
\end{prop}
\begin{proof}
Let us now consider a point $q_0$ where $T^\lambda_tu$ is twice differentiable. Then the infimum in Equation (\ref{EminDLO}) is attained at a unique solution $\gamma:[-t, 0]\rightarrow M$ for (\ref{ELAgD}) and we have
$$\frac{\partial L}{\partial v}(\gamma(0), \dot\gamma(0))=dT_t^\lambda u(q_0).$$
As $\gamma$ is minimizing, $(\varphi_{s}^\lambda(q_0, dT^\lambda_tu(q_0)))_{s\in [-t, 0]}$ has no conjugate points. \\
Because of the definition of the semi-group in (\ref{EminDLO}), we have
$$T_t^\lambda u(q_0)- {e^{-\lambda t}} u(\gamma(-t))=a_t^\lambda(\gamma(-t), \gamma(0))+\alpha(H)t$$
and
$$ \forall~ q\in M, \qquad T_t^\lambda u(q)-{e^{-\lambda t}} u(\gamma(-t))\leq a_t^\lambda(\gamma(-t), q)+\alpha(H)t.$$
Subtracting these two equations, we deduce
$$T_t^\lambda u(q)-T_t^\lambda u(q_0)\leq a_t^\lambda(\gamma(-t), q)-a_t^\lambda(\gamma(-t), q_0).$$
These two functions vanish for $q=q_0$ and have the same derivative $\frac{\partial L}{\partial v} (q_0, \dot\gamma(0))$ at $q_0$.  If we succeed in proving that 
\begin{equation}\label{EGreenD2}\frac{\partial {^2}a_t^{\lambda}}{\partial q_1{^2}}(\gamma(-t), q_0)=\cH(G^\lambda_t(q_0, dT^\lambda_tu(q_0))),\end{equation} 
we will deduce that
$$D^2T_t^{\lambda} u(q_0)\leq \cH(G^\lambda_t(q_0, dT^\lambda_tu(q_0))).$$

The arguments that we use to prove Equality (\ref{EGreenD2}) are similar to the ones given in \cite{Arnaud12}.
\begin{lemma}\label{Limvert}
For every $t>0$ and every $q\in M$, the function $v_q^t=a^\lambda_t(q,\cdot)$ is semi-concave, and satisfies
$$\cG(dv_q^t) \subset \varphi_t^\lambda(T_q^*M).$$

\end{lemma}
\begin{proof}
Because  $a_t^\lambda$ is semi-concave, the    function $v^t_q $ is semi-concave and then Lipschitz.  By Rademacher's theorem $v^t_q $ is differentiable almost everywhere.\\
  Moreover, if $q_0$ is a point where $v^t_q$ is differentiable, then $v^t_q$ has exactly one super-differential at this point, there is  only one minimizing arc $\eta$ joining $(-t, q)$ to $(0, q_0)$, and we have: 
\begin{enumerate}
\item[$\bullet$] $dv_q^t(q_0)=\frac{\partial L}{\partial v}(\eta (0), \dot\eta (0))$;
\item[$\bullet$] $(\eta(-t),     \frac{\partial L}{\partial v}(\eta (-t), \dot\eta (-t)))=(q,   \frac{\partial L}{\partial v}(\eta (-t), \dot\eta (-t)))\in T^*_q M$;
\item[$\bullet$] $\varphi^\lambda_t \left( q,       \frac{\partial L}{\partial v}(\eta (-t), \dot\eta (-t))\right)=(\eta (0), \frac{\partial L}{\partial v}(\eta (0), \dot\eta (0))) =(q_0, dv_q^t(q_0))$.
\end{enumerate}
Then we have proved that: $\varphi^\lambda_t({T_q^*M})\supset \cG (dv_q^t)$. Hence, we have selected a pseudograph in the image $\varphi_t^\lambda(T_q^*M)$ of the vertical.
\end{proof}
We come back to the point    $q_0$ where $T^\lambda_tu$ is twice differentiable and recall that $(\varphi_{s}^\lambda(q_0, dT^\lambda_tu(q_0)))_{s\in [-t, 0]}=(q_s, p_s)_{s\in  [-t, 0]}$ has no conjugate point because it is minimizing and that $\gamma =(\pi\circ\varphi_{s}^\lambda(q_0, dT^\lambda_tu(q_0)))_{s\in [-t, 0]}$ is the unique minimizing arc joining $\gamma(-t)$ to $\gamma(0)=q_0$. 
\begin{lemma}
There exists a neighborhood $V_0$ of $q_0=\gamma (0)$ in $M$ such that $v_{\gamma(-t)}^t$  is as regular as $H$ is (then at least $C^2$) and then
$$D^2v_{\gamma(-t)}^t(q_0)=\cH(G^\lambda_t(q_0, dT^\lambda_tu(q_0))).$$
\end{lemma}
\begin{proof}
Lemma \ref{Limvert} proves that $\cG (dv_{\gamma(-t)}^t)\subset \varphi^\lambda_t(T^*_{\gamma(-t)}M)$. Let us now prove that   $v_{\gamma(-t)}^t$ is smooth near $q_0$.

We use now   the so-called ``a priori compactness lemma''   (see Corollary \ref{Capcomp}) that says to us that there exists a constant $K_t=K>0$ such that  the velocities $(\dot\gamma (s))_{s\in[0, t]}$ of any minimizing arc $\gamma$ between any points $q\in M$ and $q'\in M$ are bounded by $K$; hence if we denote  by $\cK$ the set of the minimizing arcs that are parametrized by $[-t,0]$, $\cK$ is  a compact set for the $C^1$ topology because it is the image by the projection $\pi$ of a closed set of bounded orbits. Let us denote  by $\cK_0$ the set of $\eta\in \cK$ such that $\eta (-t)=\gamma(-t)$; then $\cK_0$ is compact. Let us introduce another notation: 
$\cK (q)=\{ \eta\in \cK_0 : \eta (0)=q\}$. Then $\cK (q_0)=\{\gamma\}$ and hence, because $\cK_0$ is closed,  for $q$ close enough to $q_0$, all the elements of $\cK (q)$ are $C^1$ close to $\gamma$.

Moreover, $\varphi^\lambda_t(T^*_{\gamma(-t)}M)$ is a submanifold of $T^*M$ that contains $$(q_0, \frac{\partial L}{\partial v}(q_0, \dot\gamma(0)))=(q_0, p_0).$$ Its tangent space at $(q_0, p_0)$ is $G^\lambda_t(q_0, p_0)$, which is transverse to the vertical  because $(q_s,p_s)_{s\in [-t, 0]}$ has no conjugate vectors. Hence, the manifold $\varphi^\lambda_t(T^*_{\gamma(-t)}M)$ is, in a neighborhood $U_0$ of $(q_0, p_0)$, the graph of a $C^1$ section of $T^*M$ defined  on a neighborhood $V_0$ of $q_0$ in $M$.  Moreover, because this submanifold is Lagrangian (indeed, $T^*_{\gamma(-t)}M$ is Lagrangian and $\varphi^\lambda_t$ is conformally symplectic), it is the graph of $du_0$  where $u_0 : V_0\rightarrow \R$ is a $C^2$ function. 

Now, if $q$ is close enough to $q_0$, we know that all the elements $\eta$ of $\cK (q)$ are $C^1$ close to $\gamma$, and then that $(q, \frac{\partial L}{\partial v}(\eta (0), \dot\eta (0)))$ belongs to the neighborhood $U_0$ of  $(q_0, p_0)= (q_0, \frac{\partial L}{\partial v}(\gamma (0), \dot\gamma (0)))$ and to $\varphi^\lambda_t(T^*_{\gamma(-t)}M)$. Because $\varphi^\lambda_t(T^*_{\gamma(-t)}M)\cap U_0$ is a graph,   this element is unique: $\cK (q)$ has only one element and $v_{\gamma(-t)}^t$ is differentiable at $q$, with $dv_{\gamma(-t)}^t(q)= \frac{\partial L}{\partial v}(\eta (0), \dot\eta (0)) = du_0(q)$. We deduce that near $q_0$, on the set of differentiability of $v_{\gamma(-t)}^t$, $dv_{\gamma(-t)}^t$ is equal to $du_0$; because $v_{\gamma(-t)}^t$ and $u_0$ are Lipschitz on $V_0$ and their differentials are equal almost everywhere, we deduce that on $V_0$, $v_{\gamma(-t)}^t-u_0$ is constant. Hence, on a neighborhood $V_0$ of $q_0$, $v_{\gamma(-t)}^t$ is $C^2$ and 
$$D^2v_{\gamma(-t)}^t(q_0)=\cH(G_t^\lambda(q_0, dT^\lambda_tu(q_0))).\qedhere$$
\end{proof}

\end{proof}

\subsection{On the dynamical criterion in the Hamiltonian case} \label{ssGreencrit}

We recall here two dynamical criteria concerning the Green bundles that are proven in \cite{Arnaud08}.
\begin{prop}\label{PGreencrit} {\rm (Proposition 3.12 in \cite{Arnaud08})} Let $x\in T^*M$ be a point whose negative orbit under the Tonelli Hamiltonian flow $(\varphi_t^H)$ has no conjugate points and let $v\in T_x(T^*M)$ be a tangent vector. Then, if $v\notin G_+(x)$, we have
$$\lim_{t\rightarrow +\infty}\| D(\pi\circ \varphi_{-t}^H)v\|=+\infty.$$
\end{prop}

When moreover we pay attention to points that are far from the critical points of $H$, we can use a symplectic reduction on the level  $\cE$ of $H$ in a neighborhood of  such points by using the canonical projection $p_x: T_x\cE\rightarrow T_x\cE/\R.X_H(x)$. The following statement can be deduced from Proposition 3.17 in \cite{Arnaud08}.

\begin{prop} \label{PcritGreenrelat}
Let $x\in T^*M$ be a point whose negative orbit under the Tonelli Hamiltonian flow $(\varphi_t^H)$ has no conjugate points and let $v\in T_x(T^*M)$ be a tangent vector. We assume that $(t_n)_{n\in \N}$ is a sequence of positive real numbers tending to $+\infty$ such that the angle of $X_H(\varphi^H_{-t_n}(x))$ with\\ $\ker D\pi(\varphi^H_{-t_n}(x)){ =V(\varphi_{-t_n}^H(x))}$ is uniformly bounded from below by some positive constant. Then, if $v\notin G_+(x)$, we have
$$\lim_{t\rightarrow +\infty}\| p_{\varphi_{-t}(x)}\circ D \varphi_{-t_n}^Hv\|=+\infty.$$
\end{prop}

\section{Examples and counter-examples}\label{Sexamples}
\subsection{Examples of upper and lower Green regular weak K.A.M. solutions} 
{The following proposition is proven in \cite{Arnaud2014}. It can also be deduced from the dynamical criterion and Proposition 4.12 of \cite{Arnaud08}.
\begin{prop}
Assume that   $H:T^*\T^d\rightarrow \R$ is a Tonelli Hamiltonian that has a $C^{1,1}$ weak K.A.M. solution\footnote{Observe that it is proved in \cite{Fathi2008} (Theorem 4.11.5) that every $C^1$ weak K.A.M. solution is in fact $C^{1, 1}$.} $u:\T^d\rightarrow \R$ such that there exists $t>0$ for which $\varphi_t^H$ is bi-Lipschitz conjugate to some rotation of $\T^d$. Then $u$ is upper and lower Green regular.
\end{prop}

The ideas of the proof of the following proposition are contained in \cite{Arnaud2014}. Let us recall that a vector field is Kupka-Smale if all its periodic and fixed points are hyperbolic.
\begin{prop}
Assume that   $H:T^*\T^d\rightarrow \R$ is a Tonelli Hamiltonian that has a $C^2$ weak K.A.M. solution $u: \T^d\rightarrow \R$ such that $X_{H|\cG(du)}$ is  Kupka-Smale. Then $u$ is upper and lower Green regular.
\end{prop}
\begin{proof}
 We denote by $\Oc_1, \dots , \Oc_m$ the periodic {(eventually critical)  orbits} that are contained in $\cG(du)$ and by  $W^s(\Oc_i, (\varphi_{t|\cG(du)}^H))$ and $W^u(\Oc_i, (\varphi_{t|\cG(du)}^H))$ their stable and unstable manifolds.

Because the non-wandering  set of $(\varphi^H_{t|\cG(du)})$ is $\Oc_1\cup\dots \cup\Oc_m$, then $$\displaystyle{\cG(du)=\bigcup_{1\leq i, j\leq n} \left( W^s (\Oc_j,(\varphi^H_{t|\cG(du)}))\cap W^u(\Oc_i,(\varphi^H_{t|\cG(du)}))\right)}.$$
  If $\Oc_i$ is not an attractive orbit for $(\varphi^H_{t|\cG(du)})$ then $W^s(x_i,(\varphi^H_{t|\cG(du)})$ is an immersed manifold whose dimension is less that $d$ and then has zero Lebesgue measure. We deduce that there is a dense set $D$ in $\cG(du)$ such that for all $x\in D$, $\varphi^H_t(x)$ tends to a repulsive periodic orbit when $t$ tends to $-\infty$ and tends to an attractive periodic orbit when $t$ tends to $+\infty$.

 Let us consider $x\in D$.

We assume that $(\varphi^H_t(x))$ tends to a critical attractive fixed point $x_0$ when $t$ tends to $+\infty$. We can choose $k\in ]0, 1[$ and a Riemannian metric such that in a neighborhood $\cV$ of $x_0$: $\left\|D\varphi^H_{1|\cG(du)}(y)\right\|\leq k{, ~ \forall y\in \cV}$. If $t\geq T$ is great enough, $\varphi^H_t(x)$ belongs to $\cV$ and $\left\| D\varphi^H_{1|\cG(du)}(\varphi^H_t (x))\right\|\leq k$. We deduce:
$$\forall n\in \N, \| D\varphi^H_{T+n}(x)\|\leq \| D\varphi^H_T(x)\|\prod_{i=0}^{n-1}\| D\varphi^H_{1|\cG(du)}(\varphi^H_{T+i}(x))\|\leq \| D\varphi^H_{T|\cG(du)}(\varphi^H_T(x))\|k^n;$$
hence the sequence $(D\varphi^H_{T+n}(x))_{n\in\N}$ is bounded. 

If  $(\varphi^H_t (x))$ tends to a true attractive periodic orbit $\Oc$, then $\Oc$ is a normally hyperbolic (attractive) submanifold for $(\varphi^H_{t|\cG(du)})$. Then there exists $x_0\in \Oc$ such that $x\in W^s(x_0, \varphi^H_{t|\cG(du)})$ (see for example \cite{HPS1977}). Any vector $w$ of $T_x\cG(du)$ can be written as the sum of $\lambda X_H(x)$ where $X_H$ is the Hamiltonian vector field and a vector { $v$ } tangent  to $W^s(x_0, \varphi^H_{t|\cG(du)})$. Then $D\varphi^H_t(x)X_H(x)=X_H(\varphi^H_t (x))$ is bounded and $D\varphi^H_t(x)v$ tends to $0$ when $t$ tends to $+\infty$. Finally, the family  {$(D\varphi^H_{ t|\cG(du)}(x))_{t>0}$} is bounded. By the dynamical criterion, this implies that
$$\forall x\in \cG(du), \quad T_x\cG(du)=G_+(x)$$
and then $u$ is upper Green regular.
\end{proof}

}

The following result is more or less proven in \cite{Arnaud08} (see Proposition~4.18, the statement is different but the proof is similar).

 \begin{prop}
Assume that   $H:T^*\T^2\rightarrow \R$ is a Tonelli Hamiltonian that has a $C^1$ weak K.A.M. solution $u:\T^2\rightarrow \R$ such that all the critical points of $H$ that are contained in the graph of $du$ are hyperbolic for the Hamiltonian flow. Then $u$ is upper Green regular.
\end{prop}
\begin{proof}
As the critical points of $H$ contained in $\cG(du)$ are hyperbolic for the Hamiltonian flow, their set $\cS=\{ s_1, \dots , s_n\}$ is finite.\\
We denote by $W^u$ the union of the unstable sets of the critical points of $H$ in $\cG(du)$
$$W^u=\bigcup_{i=1}^n W^u(s_i)\cap \cG(du).$$
Observe that $W^u$ and $R=\cG(du)\backslash W^u$ are measurable sets. The strategy is then to show that at Lebesgue almost every $q$ in $\pi(W^u)$ and $\pi (R)$, we have $T_{(q, du(q))}\cG(du)=G_+(q, du(q))$. We will conclude that $u$ is upper Green regular.

{\bf Case of $\pi(R)$.} For every $i\in [1, n]{\cap\mathbb{Z}}$, we construct a decreasing sequence $(D_k(i))_{k\in\N}$ of open discs that are centered at $\pi(s_i)$ and denote by $(\tilde D_k(i))_{k\in\N}$ its lift to $\cG(du)$. We denote by $f=\varphi_{-1}^H$ the time -1 flow. We also use the notations $\tilde D_k=\bigcup_{i=1}^n\tilde D_k(i)$ and $R_k=R\backslash  \tilde D_k$. For every $x\in R$ and $m\geq 1$, we introduce the notation
$$n_m(x)=\min\{ n\geq 1; f^n(x)\in R_{k}\}$$ and $F_m^k(x)=f^{n_m(x)}(x)$ when $n_m(x)\not=+\infty$. 
Then, for every $x\in R$, there exists $k_0\geq 0$ such that  the $F_k^m(x)$ are defined for every $m\geq 1$ and every $k\geq k_0$. Hence, if $E_k$ is the set of elements of $R_k$ for which $F_k^m$ is defined for every $m$, we have
$$\bigcup_{k=0}^ \infty E_k=R.$$
We know by \cite{Fathi2003} that $du$ is Lipschitz and then differentiable Lebesgue almost everywhere by Rademacher Theorem. Then if we use the notations 
$E'_k=\{ q\in \pi(E_k); D^2u(q)\quad{\rm exists}\}$ and $\tilde E'_k=\{ (q, du(q)); q\in E'_k\}$, we know that $\displaystyle{ \bigcup_{k\in\N}E'_k}$ has full Lebesgue measure into $\pi(R)$. \\
We have then
$${\rm Leb}(\pi(R_k))\geq {\rm Leb}(\pi(F_k^m(\tilde E'_k)))=\int_{E'_k}d\left(\pi\circ F_k^m( \cdot ,du(\cdot))\right)^*{\rm Leb},$$
and so
$$\begin{matrix}{\rm Leb}(\pi(R_k))&\geq \int_{E'_k}\left| {\rm det}\left(D(\pi\circ F_k^m( \cdot ,du(\cdot))\right)\right| d{\rm Leb}\\
&= \int_{E'_k}\left| {\rm det}\left(D(\pi\circ\varphi^H_{-n_m})(\cdot ,D^2u)\right)\right| d{\rm Leb}.\end{matrix}
$$
We deduce from Fatou lemma that at Lebesgue almost everywhere point $q$ in $E'_k$, we have
\begin{equation}\label{Inegdet}\liminf_{m\rightarrow+\infty }\left| {\rm det}\left(D(\pi\circ\varphi^H_{-n_m})(q, du(q))\right)_{T_{(q, du(q))}\cG(du)}\right| <+\infty.\end{equation}Using the definition of $R_k$, let us note that there exists a constant $C_k$ such that 
\begin{equation}\label{Inegchamp}\forall x, y\in R_k, \frac{1}{C_k}\leq \frac{\| X_H(x)\|}{\| X_H(y)\|}\leq C_k.\end{equation}
We then use a symplectic reduction on the energy level of $x$ by $X_H$ as explained in Subsection \ref{ssGreencrit}. 

Let us denote by $\ell$ a Lipschitz constant for $du$. Observe that
\begin{equation}\label{Edet}
\left| {\rm det}\left(D(\pi\circ\varphi^H_{-n_m})(\cdot ,D^2u)\right)\right|\geq
 \frac{1}{1+\ell}\left| {\rm det}\left(D\varphi^H_{-n_m T\cG(du)}\right)\right|
\end{equation}
We deduce from equations (\ref{Inegdet}), (\ref{Inegchamp}) and (\ref{Edet}) that
\begin{equation}\liminf_{m\rightarrow+\infty }\left\| \left(p_{\varphi_{-n_m(q, du(q))}^H}\circ D \varphi^H_{-n_m})(q, du(q))\right)_{T_{(q, du(q))}\cG(du)}\right\|<+\infty.\end{equation}
Using Proposition \ref{PcritGreenrelat}, we deduce that $T_{(q, du(q))}\cG(du)\subset G_+(q, du(q))$ and then  $T_{(q, du(q))}\cG(du)=G_+(q, du(q))$.

{\bf Case of $\pi(W^u)$.} We denote by $W^u_{\rm loc}$ the intersection of $\cG(du)$ with the union of the local unstable submanifolds of the $s_i$. If $x\in W^u$,  there exists a positive $T>0$ such that $x\in \varphi_T^H(W^u_{\rm loc})$. Then we have two cases.\\
1) We say that $x$ is {\em simple} if there exists a neighborhood $U_x$ of $x$ in $\cG(du)$ such that the only points of $U_x\cap\varphi_T^H(W^u_{\rm loc})$ are on the orbit of $x$. Observe that the set of simple orbits is countable and thus the projection of the set of simple points has zero Lebesgue measure.\\
2) Let $W'$ be the set of non simple points of $W^u$ at which $\cG(du)$ has a tangent subspace. The projection of this set has full Lebesgue measure in $\pi(W^u)$. If $x\in W'$, then for some $i$ we have $x\in \varphi_T^H(W^u(s_i))$ and because $x$ is not simple, we deduce that $T_x\cG(du)=T_x W^u(s_i)$. As $T_x W^u(s_i)=G_+(x)$, we obtain the wanted result.
\end{proof}

\subsection{An example where the convergence is not $C^2$-uniform}\label{ssnotunif}
We will show that for the pendulum, the dependance on the cohomology class is not continuous for the $C^2$ uniform topology. The Hamiltonian is given by
$$H(q, p)=\frac{1}{2}p^2 +\cos(2\pi q).$$
We use the notation $I_+=\int_0^1\sqrt{2(1-\cos(2\pi q))}dq$.  Then the map\\
 $c:[1, +\infty)\rightarrow [I_+, \infty)$ defined by 
$$c(e)=\int_0^1\sqrt{2(e-\cos(2\pi q))}dq.$$ is a homeomorphism and even a diffeomorphism when restricted to $(1, +\infty)$. \\
For every $I\in [I_+, +\infty)$, the function
$$u_I(q)=\int_0^q\left( \sqrt{2(c^{-1}(I)-\cos(2\pi s))}-I\right)ds$$
is the unique (up to the addition of a constant) weak K.A.M. solution for $T^{I}$. Observe that every $u_I$ is $C^1$.\\
Moreover, $u'_{I_+}$ is smooth on $(0, 1)$ and  because of the dynamical criterion in Proposition \ref{PGreencrit}, we have for every $q\in (0, 1)$
$$\R(1, u_{I_+}^{\prime\prime} (q))=T_{(q, I_++u'_{I_+}(q))}W^u(0, 0)=G_+(q, I_++u'_{I_+}(q)).$$
Hence $u_{I_+}$ is upper Green regular (and also Green lower regular).\\
There exists $\alpha\in (0, 1)$ such that for every $q\in (0, \alpha)$, we have 
$$u_{I_+}^{\prime\prime}(q)={2\pi}\frac{\sin(2\pi q)}{\sqrt{2(1-\cos(2\pi q))}}>\frac{1}{2}.$$
For $I>I_+$, $u_I$ is smooth and $u_I'(q)= \sqrt{2(c^{-1}(I)-\cos(2\pi q))}-I$ attains its minimum at $q=0$ where $u_I^{\prime\prime}(0)=0$. Hence there exists $\alpha(I)\in (0, \alpha)$ such that
$$\forall q\in [0, \alpha(I)], u_I^{\prime\prime}(q)<\frac{1}{4}.$$
We then deduce 
$$\forall I>I_+, \quad \| u_I^{\prime\prime}-u_{I_+}^{\prime\prime}\|_{L^\infty}\geq  \left\| (u_I^{\prime\prime}-u_{I_+}^{\prime\prime})\big|_{(0, \alpha(I))}\right\| \geq  \frac{1}{4}.$$
We don't have continuous dependence of $u_I$ on $I$ for the uniform $C^2$ distance.

\subsection{Examples of  weak K.A.M. solutions that are not upper Green regular nor lower Green regular and to which the Lax-Oleinik semi-group doesn't $d_{2,1}$-converge.}\label{ssnotup}
{Let $\cS$ be a closed surface with negative curvature. Let us denote by $M=T^1\cS$ its unitary tangent bundle and by $X$ the geodesic vector field. We then consider the Ma\~ n\'e Lagrangian (see \cite{Mane1992}) $L:TM\rightarrow \R$ that is defined by
$$L(q, v)=\frac{1}{2}\| v-X(q )\|^2.$$
The corresponding Hamiltonian is given by
$$H(q, p)=\frac{1}{2}\| p\|^2+p.X(q).$$
Observe that the critical level is $\cH=\{ H=0\}$ because this level contains an exact Lagrangian graph (see \cite{Fathi2008}).\\
Then $0$ is a weak K.A.M. solution. We denote by $Z$ the zero section in $T^*M$. The set  $Z$ is hyperbolic for the restriction of $(\varphi_t^H)$ to the energy level $\cH=\{ H=0\}$. We denote by $E^s$, $E^u$ the 3-dimensional stable and unstable bundles along $Z$: they contain the vector field direction and also the strong stable (unstable) bundle. By \cite{Arnaud12}, we have $E^u(x)=G_+(x)$ and $E^s(x)=G_-(x)$. As {$(\varphi^H_{t|Z})$} is Anosov, the intersection of $E^u(x)$ (resp. $E^s(x)$) with $T_xZ$ is 2-dimensional and then we have
\begin{equation}\label{Enupper} \forall x\in Z, G_+(x)\not=T_xZ.\end{equation}
So $u$ is nowhere upper Green regular.\\

 Let us now prove that $u$ is the only weak K.A.M. solution (up to the addition of a constant).\\
As the flow $(\psi_t)$ of $X$ is transitive, the projected Aubry set for $H$ is the whole {$M$}. To prove that, we use the characterization of the projected Aubry set that is given in \cite{Fathi2008}. Let $q_0\in{ M}$ be any point. As $(\psi_t)$ is transitive, for every neighborhood $V$ of $q_0$ and any $T>0$, there exist $q\in V$ and $t\geq T$ such that $q, \psi_t(q)\in V$. Let $\gamma:[0, t+\varepsilon]\rightarrow {M}$ be the closed arc that is made with the three following pieces.
\begin{enumerate}
\item the straight segment that joins $q_0$ to $q$ with unitary derivative;
\item the arc of orbit $(\psi_sq)_{s\in [0, t]}$;
\item the straight segment that joins $\psi_t(q)$ to $q_0$ with unitary derivative.
\end{enumerate}
The Lagrangian action of the first and third parts of this arc are very small, and the second one is zero because we have a piece of orbit. Hence the action of $\gamma$ can be very small. Hence $q_0$ belongs to the Aubry set.\\
This implies that, up to the addition of a constant, there is only one weak K.A.M. solution, and so the only weak K.A.M. solutions are the constant functions.\\

We will now build an example of an initial condition $u$ for the Lax-Oleinik semi-group such that the conclusion of Theorem \ref{TC2convLO} is not satisfied, i.e. such that the family $(d_{2,1}( T_t u,0))_{t\in [0, +\infty[}$ doesn't tend to 0 when $t$ tends to $+\infty$.\\
We choose a large set of   points $(q_1, 0)$, \dots, $(q_n, 0)$ in $Z$, we 
 fix some $T>0$ and we introduce the following functions.

\begin{nota}
The Lagrangian action is denoted by $a_T(q_1, q_2)= \inf\int_0^TL(\gamma(t), \dot\gamma(t))dt$  where the infimum is taken over all the absolutely continuous curves $\gamma:[0, T]\rightarrow M$ such that $\gamma(0)=q_1$ and $\gamma (T)=q_2$.
\end{nota} 
The $a_T$ is semi-concave and then Lipschitz (see \cite{Bernard08}).  {Define}
\begin{itemize}
\item $u_i^T(q)=a_T(q_i, q)$;
\item $u^T(q)=\min\{ u_i^T(q); 1\leq i\leq n\}$.
\end{itemize}
All these functions are non-negative and $K_T$-semi-concave. By Lemma \ref{Limvert}, we have 
$$\forall i\in \{ 1, \dots, n\}, \quad  \cG(du_i^T) \subset \varphi_T^H(T_{q_i}^*M);$$
and so because of semi-concavity
\begin{equation}\label{Egraphvert} \cG(du^T) \subset \bigcup_{i=1}^n\varphi_T^H(T_{q_i}^*M).\end{equation}
Note that $u_i^T(\psi_{T}(q_i))=0$ and so for every $i\in \{ 1, \dots, n\}$, we have $u^T(\psi_{T}(q_i))=0$. Because $u^T$ is $K_T$-semi-concave, non-negative and vanishes at the points $\psi_{T}(q_i)$, if we choose the $q_i$'s in such a way that the $\psi_{T}(q_i)$ are $\epsilon$-dense in $M$ for a small $\varepsilon$,  then {$u^T$} is $C^0$ close to $0$.\\

Let us now prove that $u^T$ can be chosen such that the graph of $du^T$ is in a small  neighborhood of the zero section.  We denote by $\cM_T$ the set of $T$-minimizing orbits for the Euler-Lagrange flow $(f_t^L)_{t\in \R}$.
$$\cM_T=\left\{ (f_T^L(x))_{t\in [0, T]}; ~\left( \pi\circ f_T^L(x)\right)_{t\in [0, T]}\quad{\rm is}\quad{\rm minimizing}\right\}.$$
Observe that $\cM_T$ is compact. We can endow it as well with the $C^0$ or $C^1$ topology that are equal. We have
\begin{itemize}
\item $\forall \Gamma\in\cM_T,  \quad A_L(\Gamma)=\int_0^TL\circ \Gamma(t)dt\geq 0$;
\item $\forall \Gamma\in\cM_T,  \quad A_L(\Gamma)= 0\Leftrightarrow \Gamma([0, T])\subset \cG(X)$.
\end{itemize}

We introduce the notation 
$$\cZ_T=\{ \Gamma\in \cM_T; \Gamma([0, T])\subset \cG(X)\}.$$
Then $\cZ_T=\{ \Gamma\in \cM_T; A_L(\Gamma)=0\}$ is compact. We now fix a small neighborhood $\cN_T$ of $\cZ_T$ in $\cM_T$. We introduce
$$\varepsilon=\frac{1}{2}\inf\{ A_L(\Gamma); \Gamma\in \cM_T\backslash \cN_T\}.$$
Then $\varepsilon>0$. We choose $\alpha>0$ such that
$$\forall q, q_1, q_2\in M, d(q_1, q_2)<\alpha {\Rightarrow} |a_T(q, q_1)-a_T(q, q_2)|<\varepsilon.$$
We choose a finite number of points $x_1, \dots, x_n\in \cG(X)$ { on the graph of the vector-field $X$} such that
$$M\subset \bigcup_{i=1}^nB(\pi\circ f_T^L(x_i), \alpha).$$
and use the notation $x_i=(q_i, X(q_i))$. Then we define the $u_i^T$'s and $u^T$  as before. Let us consider $q\in M$. Then $q$ belongs to some ball $B(\pi\circ f_T^L(x_i), \alpha)$ and so we have
$$u_i^T(q)=a_T(q_i, q)-a_T(q_i, \pi\circ f^L_T(x_i))<\varepsilon.$$
We deduce that $u^T(q)\leq u^T_i(q)<\varepsilon$. Let $j\in\{ 1, \dots, n\}$ be such that $u_j^T(q)=u^T(q)$. We have $u_j^T(q)=a_T(q_j, q)<\varepsilon$. Hence for every $\Gamma\in \cM_T$ such that $\pi\circ \Gamma(0)=q_j$, $\pi\circ \Gamma(T)=q$, we have $\Gamma\in\cN_T$.\\
 If now $du^T(q)$ exists, we have $du^T(q)=du_j^T(q)$, the minimizing $\Gamma$ is unique and we denote $\gamma=\pi\circ \Gamma$, then $\dot\gamma$ is $C^0$-close to $X\circ \gamma$ because $\Gamma\in\cN_T$. By using Legendre map, this implies that $du^T_j\circ \gamma$ is $C^0$-close to the zero section and then $du^T(q)=du^T_j(\gamma(T))$ is close to zero.\\

 So we have proved that we can assume that the graph of $du^T$ is contained in a neighborhood $\cN$ of the zero section that is as small as we want. By \cite{Bernard08}, observe that 
\begin{equation}\label{EBernard1}\forall t\geq 0, T_tu^T=u^{t+T} \quad{\rm and}\quad \cG(du^{T+t})\subset \varphi_t^H(\cG(du^T)).\end{equation}
 By continuity of the flow and compacity of the closure of $\cG(du^T)$, there exists a small $\tau>0$ such that
 $$\forall t\in [0, \tau], \quad \cG(du^{T+t})\subset \varphi_t^H(\cG(du^T))\subset \cN.$$
 We now use Lemma 7 of \cite{Arnaud2005} and find some $\beta>0$ such that 
 $$\forall u\in C^0(M, \R), \quad \| u\|_\infty<\beta\Rightarrow \cG(dT_\tau u)\subset \cN.$$
We can assume that $u^T$ satifies $\| u^T\|_\infty <\beta$. Then for every $t\geq \tau$, we have 
$$\| T_{t-\tau}u^T -0\|_\infty=\|T_{t-\tau}u^T-T_{t-\tau}0\|_\infty\leq \| u^T\|_\infty<\beta$$
because of the non-expansiveness of the Lax-Oleinik semi-group (see \cite{Fathi2008}). We deduce
$$\forall t\geq \tau,\quad \cG(du^{t+T})=\cG(dT_\tau(T_{t-\tau}u^T ))\subset \cN.$$
We have then proved that 
\begin{equation} \label{EgraphN}\forall t\geq 0, \cG(dT_tu^T)\subset \cN. \end{equation}\\

Let us recall that the flow $(\psi_t)$ is Anosov. This implies that the cocycle that we will now introduce is hyperbolic on $Z$.

 The cocycle is defined in a fiber bundle over a neighborhood $\cN$ of the zero section $Z$ in $T^*M$. At a point $x\in \cN$, we consider the tangent space $T_x\cH_x$ of the energy level $\cH_x=\{ H=H(x)\}$. Then {$E_x$} is defined as being the reduced linear space $T_x\cH_x/\R. X_H(x)$ endowed with the quotient norm $\|\cdot\|$ and the corresponding projection is denoted by $p_x: T_x\cH_x\rightarrow E_x$.

 As we take the quotient of an Anosov flow by the vector field, the corresponding reduced cocycle $(M_t)$ of $(D\varphi^H_t)$ restricted to $Z$ is hyperbolic, and has an invariant splitting $E=E^s\oplus E^u$ where the stable and unstable bundles are 2-dimensional. By \cite{Yoccoz1995}, we can translate the hyperbolicity condition by using some cones. This is an open condition and we can extend these cones to a neighborhood $\cN$ of $Z$ such that
\begin{itemize}
\item there exists a continuous splitting $E=E^1\oplus E^2$ on $\cN$ that coincides with $E=E^s\oplus E^u$ on $Z$ and two norms $|\cdot|_i$ on $E^i$ such that $$C_x=\{ v=v_1+v_2, v_1\in E_x^1, v_2\in E_x^2, |v_1|_{1, x}\leq |v_2|_{2, x}\};$$
the family $(C_x)_{x\in \cN}$ is the associated cone field; the dual cone field is the family $(C^*_x)_{x\in \cN}$ defined by $C_x^*=E_x\backslash {\rm int}C_x$.
\item for some constant $c>0$, we have for every $x\in \cN$, $v_1\in E^1_x$ and $v_2\in E^2_x$
$$c^{-1}\| v_1+v_2\|\leq \max\{ |v_1|_{1, x}, |v_2|_{2,x}\}\leq c\| v_1+v_2\|_x.$$
\item there exists an integer $m\geq 1$ and two constants $\lambda, \mu>1$ so that
\begin{enumerate}
\item for $x\in \cN$, $M_1(C_x)\subset \widetilde C_{\lambda,\varphi^H_{1}(x)}$ where
$$\widetilde C_{\lambda,x}=\{ v=v_1+v_2\in E_x; \lambda|v_1|_{1, x}\leq |v_2|_{2, x}\};$$
\item for $x\in \cN$, for $v\in C_x$, $\| M_m(v)\|_{\varphi^H_m(x)}\geq \mu.\| v\|_x$;
\item for $x\in \cN$, for $v\in C^*_x$, $\| M_{-m}(v)\|_{\varphi^H_{-m}(x)}\geq \mu.\| v\|_x$.
\end{enumerate}
\end{itemize}
Following \cite{Arnaud12}, we now introduce some notations.

\begin{notas}
\begin{itemize}
\item for $x\in\cN$, we denote by {$v(x)$} the $p$-projection of the intersection of the vertical with the tangent space to the energy level $v(x)=p_x\left( T_x\cH_x\cap V(x)\right)$;
\item when $\varphi_s^H(x)\in\cN$ for every $s$ between $0$ and $-t$, we denote by $g_t(x)$ the subspace $M_t(\varphi_{-t}^Hx){v(\varphi_{-t}^Hx) }$. Moreover
\begin{itemize}
\item if $\varphi_s^H(x)\in\cN$ for every $s\in (-\infty, 0)$ and $g_u(x)$ is transverse to $v(x)$ for every $u>0$, then $g_+(x)=\lim_{t\rightarrow +\infty}g_t(x)$ exists and is a reduced Green bundle; for $x\in Z$, we have $g_+(x)=E^u(x)=E^2(x)$;
\item if $\varphi_s^H(x)\in\cN$ for every $s\in (0, +\infty)$ and $g_u(x)$ is transverse to $v(x)$ for every $u<0$, then $g_-(x)=\lim_{t\rightarrow -\infty}g_t(x)$ exists and is a reduced Green bundle; for $x\in Z$, we have $g_-(x)=E^s(x)=E^1(x)$.
\item for every $x\in Z$, we also use the notation $\cH_x=p_x(T_xZ)$ and denote by $\cH$ the corresponding bundle over $Z$.
\end{itemize}
\end{itemize}

\end{notas}
On $Z$, $g_-=E^1$ is well defined and transverse to $v(x)$. The hyperbolicity of $(M_t)$ on $Z$ implies that for every $m\geq 1$, there exists some $n>0$ such that \begin{equation}\label{EGreencone} \forall x\in Z, \quad g_n(x)=M_n(v(\varphi_{-n}^Hx))\in \widetilde C_{\lambda^{m+1}, x},\end{equation}
and we can also assume that 
\begin{equation}\label{Econecone}\forall x\in Z, \quad M_n( \widetilde C_{\lambda^m, x})\subset \widetilde  C_{\lambda^{m+1}, \varphi_n^Hx}.\end{equation}
Observe that $E^2$ is different from $\cH_x$ (because of Equation (\ref{Enupper})). Hence we can choose $m\in\N$ large enough such that 
\begin{equation}\label{nupper2}\forall x\in Z, \quad \cH_x\nsubset \widetilde C_{\lambda^m, x}.\end{equation}
We now choose an eventually smaller neighborhood $\cN$ of $Z$ that satisfies the following conditions, where we assume that we choose a metric  on $E_\cN$ that allows us to compare tangent vectors of different fibers.
\begin{equation}\label{Efarcone} \forall x\in\cN,{\exists w\in \cH_x, \quad \| w\|=1 , d(w, \widetilde C_{\lambda^m, x} ) >\varepsilon_0 } 
\end{equation}
for some $\varepsilon_0>0$ because of Equation (\ref{nupper2});
\begin{equation}\label{Evertcone}\forall x\in\cN, \quad  g_n(x)=M_n(v(\varphi_{-n}^Hx))\in \widetilde C_{\lambda^{m}, x}\end{equation}
because of Equation (\ref{EGreencone});
\begin{equation}\label{Econecone2}\forall x\in\cN, \quad M_n( \widetilde C_{\lambda^m, x})\subset \widetilde  C_{\lambda^{m}, \varphi_n^Hx} \end{equation}
because of Equation (\ref{Econecone}).

We  now choose $u^T$ depending on $\cN$ as before. We have proved that for every $t\geq 0$, $\cG(dT_tu^T)\subset \cN$ (see Equation (\ref{EgraphN})). We also have by Equation (\ref{EBernard1}) that
$$\forall t\geq 0, \forall s\in [0, t],  \quad \varphi_{s}^H\left( \cG(dT_tu^T)\right)\subset \cG(dT_{t-s}u^T)\subset\cN$$
and by Equation (\ref{Egraphvert})
\begin{equation}\label{EGreenpseudo}\forall  t>0, \forall x\in \cG(dT_tu^T), ~T_{\varphi_tx}\cG(dT_tu^T)=D\varphi_t\left( (T_x\cG(du^T)\right)=G_t(\varphi_t^Hx).\end{equation}
We deduce from Equations (\ref{EGreenpseudo}),  (\ref{Evertcone}) and (\ref{Econecone2}) that 
$$\forall k\in\N, \forall x\in \cG(dT_{nk}u^T), \exists D^2T_{nk}u^T(x)\Rightarrow p_x\left( T_x\cG(dT_{nk}u^T)\right)\subset \widetilde C_{\lambda^m, \varphi_{nk}^Hx}$$
and then by Equation (\ref{Efarcone})
$$\forall k\in\N, \forall x\in \cG(dT_{nk}u^T), \exists D^2T_{nk}u^T(x)\Rightarrow \| D^2T_{nk}u^T(\pi(x))\|\geq \varepsilon_0.
$$
and then the quantity $d_{2,1}( T_tu^T,0)$ doesn't tend to $0$ when $t$ tends to $+\infty$, hence doesn't satisfy the conclusion of Theorem \ref{TC2convLO}.
 }

\section{Proof of the $C^1$ convergence}

\subsection{Proof of Theorems  \ref{TC1discounted} and \ref{TC1cohom}} We extend the ideas that were introduced in \cite{Arnaud2005} to a more general setting.\\
We consider a map $c\in C \mapsto H_c$ defined on some compact metric space $C$ that is continuous for the $C^2$ open-compact topology and such that every $H_c: T^*M\rightarrow \R$ is a Tonelli Hamiltonian. The associated Lagrangian is denoted by $L_c$. We will denote by $(\varphi_t^{c, \lambda})$ the flow associated to the $\lambda$-discounted equation for $H_c$.
\begin{prop}\label{Ppseudog}
{ Let us fix $t>0$ and $(u_0, c_0, \lambda_0)\in C^0(M, \R)\times C\times \R_+$ with $u_0$ fixed point of the semi-group $(T^{c_0, \lambda_0}_\tau)$. For every $\varepsilon>0$, there exists a neighborhood $\cN$ of $(u_0, c_0, \lambda_0)$ such that for every $(u, c, \lambda)\in\cN$, we have
$$d_H(\overline{\cG(dT^{c,\lambda}_tu)},  \overline{\cG(du_0)})\leq \varepsilon.$$}
\end{prop}
To finish the proofs of Theorems \ref{TC1discounted} and \ref{TC1cohom}, we have to apply Proposition \ref{Ppseudog} when
\begin{itemize}
\item either the space $C$ is only one point, $u=u_\lambda$ is the discounted solution and $u_0$ the limit weak K.A.M. solution;
\item or $\lambda=0$ is fixed and $c\in H^1(M, \R)$, $u=u_c$ is a weak K.A.M. solution for the cohomology class $c$.
\end{itemize}

\begin{proof}
We recall that because of the a priori compactness Lemma (Corollary \ref{Capcomp}), there exists for every $t>0$ a compact subset $K_t\subset T^*M$ such that, for every $c\in C$ and $\lambda\in \Lambda$ where $\Lambda$ is any compact subset of $\R$, any minimizing orbit $(\varphi_s^{c, \lambda}(q, p))_{s\in [0, t]}$ takes all its values in $K_t$. Observe that if $T>t$, we can choose $K_T=K_t$.

Let us fix $t>0$. We define the map
 $\cM_t:T^*M\times C\times  \R_+\rightarrow \R$ by
 $$\cM_t(q, p, c, \lambda)=\int_{-t}^0e^{\lambda s} L_c (\pi\circ\varphi_{s}^{c, \lambda}(q, p), \frac{\partial}{\partial s}(\pi\circ\varphi_{s}^{c, \lambda}(q, p)))ds.$$
 Observe that this map is continuous with respect to all the variables.\\

 We have for every $t\geq T>0$ and every continuous map $u:M\rightarrow \R$
\begin{equation}\label{EC1discounted} T_t^{c, \lambda}u(q_0)=\min_{(q_0, p)\in K_T} \left(e^{-\lambda t}u(\pi\circ \varphi_{-t}^{c, \lambda}(q_0, p))+\cM_t(q_0, p, c, \lambda)\right).\end{equation}
  We now define for every $u\in C^0(M, \R_+)$, $(c, \lambda)\in C\times \R_+$ and $t\geq T>0$ the map
 $\cU_t(u, c, \lambda):K_T\rightarrow \R$ by
 $$\cU_t(u, c, \lambda)(q, p)=e^{-\lambda t}u(\pi\circ \varphi_{-t}^{c, \lambda}(q, p))+\cM_t(q, p, c, \lambda).$$
  Then every map $\cU_t(u, c, \lambda)$ is continuous and the map $$\cU_t:C^0(M, \R)\times C\times\R_+\rightarrow C^0(K_T, \R)$$ is itself continuous if $C^0(M, \R)$ and $C^0(K_T, \R)$ are endowed with the uniform $C^0$ distances.
 
 This implies that the map 
 $(u, c, \lambda)\in C^0(M, \R)\times C\times \Lambda\mapsto T_t^{c, \lambda}u\in C^0(M, \R)$ that is defined by(see Equation (\ref{EC1discounted}))   $$T_t^{c, \lambda}u(q_0)=\min_{(q_0, p)\in K_T}\cU_t(u, c, \lambda)(q_0, p)$$
 is also continuous.

Moreover, the  corresponding $\argmin$ function, that is the function $$\cE_{\color{blue} t}: M\times C^0(M, \R)\times  C\times \Lambda\rightarrow \cK(K_t)$$ that takes its values is the set $\cK(K_t)$ of non-empty compact subsets of $K_t$ and is defined by
 $$\cE_t(q_0,u,  c, \lambda)=\{ (q_0, p)\in T^*_{q_0}M;  T_t^{c, \lambda}u(q_0)=\cU_t(u, c, \lambda)(q_0, p)  \}$$
 is an upper-continuous function when $\cK(K_t)$  is endowed with the Hausdorff distance. 
Hence $$\cF_t(u, c, \lambda)=\bigcup_{q\in M}\cE_t(q,u,  c, \lambda)$$ is also compact. Observe that 
 $\cG(dT^{c, \lambda}_tu)\subset   \cF_t(u, c, \lambda)$  and then 
 \begin{equation}\label{Einclusgraph} \overline{\cG(dT^{c, \lambda}_tu)}\subset  \cF_t(u, c, \lambda).\end{equation}
  Let us prove that Equation (\ref{Einclusgraph}) is an equality for $u=u_0$. We recall that $u_0$ is a fixed point of the semi-group $(T^{c_0, \lambda_0}_\tau )$. But we know from Equation~{(\ref{Emingraph})} that for every $\tau\in (0, t)$, we have
 $$ \cF_t(u_0, c_0, \lambda_0)\subset  \cF_\tau(u_0, c_0, \lambda_0)\subset \varphi^{c_0, \lambda_0}_{\tau}\left(\cG(du_0)\right)\subset \varphi^{c_0, \lambda_0}_{\tau}\left(\overline{\cG(du_0)}\right)$$
 and then by taking the limit for $\tau$ tending to $0$, we deduce that 
\begin{equation}\label{Eotherinclus}  \cF_t(u_0, c_0, \lambda_0)\subset \overline{\cG(du_0)}.\end{equation}
 Equations (\ref{Einclusgraph} ) and (\ref{Eotherinclus} ) give finally 
\begin{equation} \label{Eequal} \overline{\cG(du_0)}= \cF_t(u_0, c_0, \lambda_0).\end{equation}
 
 Let us now fix $\varepsilon>0$. If $(u_0, c_0, \lambda_0, q_0)\in C^0(M, \R)\times C\times [0, 1]\times M$, there exists a neighborhood $\cV$ of $(u_0, c_0, \lambda_0)$ and a neighborhood $V$ of $q_0$ such that for every $(q,u,  c, \lambda)\in V\times \cV$, we have
 $$\cE(q,u,  c, \lambda)\subset  \cE(q_0,u,  c, \lambda)_\varepsilon,$$
 where we use the following notation for $K\subset T^*M$.
 
 \begin{nota} 
 $$K_\varepsilon =\{ x\in T^*M; d(x, K)\leq \varepsilon\}.$$
 \end{nota}
  Then we can extract a finite covering of $M$ by $(V_i)_{1\leq i\leq n}$ that  are built as before, with neighborhoods $\cV_i\times V_i$ of $(u_0, c_0, \lambda_0, q_i)$. Then $( \cE(q_i,u_0,  c_0, \lambda_0)_\varepsilon)_{1\leq i\leq n}$ is a covering of  $\cF_t(u_0, c_0, \lambda_0)=\overline{\cG(du_0)}$ by equation (\ref{Eequal} ) and we have for $\displaystyle{(u, c, \lambda)\in \cV=\bigcap_{1\leq i\leq n}\cV_i}$
  $$ \overline{\cG(dT^{c, \lambda}_tu)}\subset  \cF_t(u, c, \lambda)=\bigcup_{q\in M}\cE_t(q,u,  c, \lambda)\subset \bigcup_{i=1}^n  \cE_t(q_i,u_0,  c_0, \lambda_0)_\varepsilon\subset (\overline{\cG(du_0)})_\varepsilon.
 $$
 To obtain the wanted conclusion, we only need to prove that $\overline{\cG(du_0)}\subset  (\overline{\cG(dT^{c, \lambda}_tu)})_\varepsilon$ for $(u, c, \lambda)$ close to $(u_0,c_0, \lambda_0)$.

We denote by $\cD_0$ the set of derivability of $u_0$ and we consider a finite covering of $\overline{\cG(du_0)}$ by open balls with radius $\eta=\frac{\varepsilon}{2}$ and centers at $(q_i, du_0(q_i))_{1\leq i\leq n}$ where $q_i\in \cD_0$. As the map $\cE_t$ is upper semi-continuous and has for value at every $(q_i, u_0, c_0, \lambda_0)$ the  set $\{ (q_i, du_0(q_i))\}$, there exists $\alpha>0$ such that
$$\begin{matrix}\| u-u_0\|_{c_0}<\alpha, d(c,c_0)<\alpha, &|\lambda-\lambda_0|<\alpha, d(Q_i, q_i)<\alpha\hfill\\
&\Rightarrow d((q_i, du_0(q_i)), \cE_t(Q_i, u ,c, \lambda))<\varepsilon_0.\end{matrix}$$
Then we choose for every $i$ a $Q_i$ where $u$ is differentiable and obtain
$$d((q_i, du_0(q_i)), (Q_i, dT_t^{c, \lambda}u(Q_i))<\varepsilon_0.$$
Then we have
$$\begin{matrix}\forall q\in \cD_0, \exists i\in [1, n],  d((q, du_0(q)), (Q_i, dT_t^{c, \lambda}u(Q_i)) \hfill\\
\leq d((q, du_0(q)), (q_i, du_0(q_i))+d((q_i, du_0(q_i)), (Q_i, dT_t^{c, \lambda}u(Q_i))\leq 2\varepsilon_0=\varepsilon.\end{matrix}$$

\end{proof}

\subsection{Hausdorff distance in $T^*M$ and $C^1$ convergence}
We will prove a proposition that implies that if a family of pseudographs $(\cG(\eta_{c_\lambda}+du'_\lambda))_{\lambda\in\Lambda}$ converges  to the graph of $\cG(\eta_c+du)$ for the Hausdorff distance and if the map $u:M\rightarrow \R$ is $C^1$, then the derivatives $(du_\lambda)$ $C^0$- uniformly converge to $du$. Hence Corollaries \ref{CC1discounted} and \ref{CC1cohom} comes easily from Theorems \ref{TC1discounted} and \ref{TC1cohom}.

\begin{nota}
If $K\subset T^*M$ and $q\in M$, we will denote by $K_q=K\cap T_q^*M$.
\end{nota}
\begin{prop}\label{PC1graph}
Let us consider a family $(K^\lambda)_{\lambda\in\Lambda}$ of compact subsets of $T^*M$ and let $\cG=\cG(\eta)\subset T^*M$ be the graph of a continuous map $\eta$ defined on the whole $M$. Assume that $\displaystyle{\lim_{\lambda\rightarrow \lambda_0}d_H(K^\lambda, \cG)=0}$. Then
$$\lim_{\lambda\rightarrow \lambda_0}\sup_{q\in M} \sup_{x\in K^\lambda_q} d(\eta(q), x)=0.$$
\end{prop}
\begin{proof}
Assume that the result is not true. Then there exists a sequence $(\lambda_n)$ that converges to $\lambda_0$  and an $\varepsilon>0$ such that
\begin{equation}\label{Eblabla}\forall~ n\in\N, \exists~ q_n\in M, \exists~ x_n\in K^{\lambda_n}_{q_n}, \quad d(\eta(q_n),x_n)\geq \varepsilon.\end{equation}
Extracting a subsequence, we can assume that $(q_n)$ converges to some $q_0$ in $M$. For $n\geq N$ large enough we have 
$${\sup_{x\in K^{\lambda_n}} d( x, \cG)}\leq 1.$$
Hence for $n\geq N$, we have $d(x_n, \cG(\eta))\leq 1$, which means that $x_n$ takes its values in a fixed compact set. Extracting a subsequence, we can then assume that $(x_n)$ converges to some $x\in T^*M$. We deduce from equation  \eqref{Eblabla}   and continuity of $du$  that $d(\eta(q), x)\geq \varepsilon$. \\
This implies that $x\notin \cG(\eta)$. Let us recall that the graph of a continuous map is closed. Hence there exists some $\beta>0$ such that $B(x, 2\beta)\cap \cG(\eta)=\emptyset$. \\
As $(x_n)$ converges to $x$, for $n\geq N'$ large enough, we have $x_n\in B(x, { \beta})$ and then $B(x_n, { \beta})\cap \cG(\eta)=\emptyset$. Hence we obtain finally
$$\forall n\geq N', x_n\in K^{\lambda_n}\quad{\rm and}\quad d(x_n, \cG(\eta))\geq \beta;
$$
and then 
$$\forall n\geq N', d_H( K^{\lambda_n}, \cG(\eta))\geq \beta;$$
which contradicts the hypothesis.
\end{proof}

\section{Proof of the $C^2$ convergence}

We give a proof that is valid for Theorems \ref{TC2discounted},  \ref{TC2convLO} and \ref{TC2cohomclass}.\\
We fix $u_\infty$ the $C^1$ and upper Green regular weak K.A.M. solution. It is proved in \cite{Fathi2008} that any $C^1$ weak K.A.M. solution is $C^{1,1}$, so $u_\infty$ is $C^{1,1}$ and then semi-concave and semi-convex.\\
We recall that we consider a family of semi-concave functions $u$ that
\begin{itemize}
\item converges to $u_\infty$ in $C^1$ uniform topology; this comes from Corollaries  \ref{CC1discounted} and \ref{CC1cohom} and also  \cite{Arnaud2005} joint with  Proposition \ref{PC1graph};
 \item satisfies the following lemma.
\end{itemize}

\begin{lemma}\label{D2Lebestimate}
For $\varepsilon>0$ small enough,  we have
\[
\text{Leb}\{ \theta\in \mathbb{T}^d ~:~ D^2 u - D^2 u_\infty \nless  \varepsilon {\bf 1} \} < \varepsilon
\]
where ${\bf 1}$ is the identity matrix of the standard scalar product.
\end{lemma}
\begin{proof}
Since $u_{\infty}$ is semi-concave, $D^2 u_\infty: \mathbb{T}^d \rightarrow\mathcal{S}_d$ is defined (and measurable) Lebesgue almost everywhere. Due to Lusin's theorem, there exists a compact $\mathcal{K}\subset \mathbb{T}^d$ with $\text{Leb}(\mathbb{T}^d\setminus \mathcal{K}) < \frac{\varepsilon}{10}$ such that $D^2u_\infty\big|_{\mathcal{K}}$ is continuous. Hence, there exists $\alpha >0$, such that for any $\theta, \theta' \in \mathcal{K}$ with $d(\theta, \theta')< \alpha$ we have 
\begin{equation}\label{D2estimate}
\| D^2u_\infty(\theta) - D^2 u_\infty(\theta')\| < \frac{\varepsilon}{2}.
\end{equation}

Let us recall that the approximated solutions $u$ that we consider is semi-concave, $C^0$-close to  $u_\infty$ and is of one of the three possible kinds that we now describe.\begin{itemize}
\item $u=u^\lambda$ that is a discounted solution for a small $\lambda$; then by Corollary\ref{CC1discounted}, $u$ is $C^1$ close to $u_\infty$;
\item $u=T_tU$ for some $U\in C^0(\T^d, \R)$ and some $t>0$ large enough; then by Theorem 1 of \cite{Arnaud2005} and Proposition \ref{PC1graph}, $u$ is $C^1$ close to $u_\infty$;
\item $u=u_c$ that is a weak K.A.M. solution for a cohomology class $c$ close to $0$; then by Corollary~\ref{CC1cohom}, $u$ is $C^1$ close to $u_\infty$.
\end{itemize}
We deduce that for
every $x =(\theta, du_\infty(\theta) ) \in {\cG}(du_{\infty|\mathcal{K}})$, due to Propositions \ref{PupsemicGreen} and \ref{tLOGreen}, there exists $\eta_x$ with $\eta_x< \alpha$ such that for every $y\in B(\theta, \eta_x)\times B(du_\infty(\theta), \eta_x) $, we have  for some $t>0$ large enough
\begin{equation}\label{Greenbundleestimate}
D^2 u(\pi(y)) < \mathcal{H} (G_t(y))  < D^2 u_\infty(\theta) + \frac{\varepsilon}{2} {\bf 1}  = \mathcal{H}(G_+(x)) +  \frac{\varepsilon}{2} {\bf 1}.
\end{equation}

 Since $\mathcal{G}(du_\infty\big|_\mathcal{K})$ is compact, there exists $n$ and $x_i\in \mathcal{K}$ with $i=1, \cdots, n$  such that
 \[
 \mathcal{G}(du_\infty\big|_\mathcal{K}) \subset \bigcup_{i=1}^n B(\theta_i, \eta_{x_i}) \times B(du_\infty(\theta_i), \eta_{x_i}):=\mathscr{N}
 \]
 
 Because $u$ converges to $u_\infty$ in $C^1$ uniform topology, one can choose $u$ such that $\mathcal{G}(du\big|_\mathcal{K}) \subset \mathscr{N}$. For every $\theta\in \mathcal{K}$,  without loss of generality, we assume $(\theta, du(\theta)) \in B(\theta_1, \eta_{x_1})\times B(du(\theta_1), \eta_{x_1})$. Using \eqref{D2estimate} and \eqref{Greenbundleestimate}, we obtain
 \[
 D^2 u(\theta) < D^2u_\infty(\theta_1) + \frac{\varepsilon}{2} {\bf 1} <  D^2u_\infty(\theta) + \varepsilon {\bf 1}.  \qedhere
 \]
\end{proof}

We fix one integer $i\in\{1, \dots, d\}$ and we consider one (non-injective) arc $\gamma:t\in[0, {3}]\rightarrow \gamma(t)\in\T^d$ defined by $t=\theta_i\in[0, {3}]$ and the other coordinates $\theta_j$ fixed. Because $u$ and $u_\infty$ are semi-concave, they are differentiable Lebesgue almost everywhere and admits a second derivative Lebesgue almost everywhere. By Fubini Theorem,  there exists some $\sigma\in [0, 1)$ such that  for Lebesgue almost every choice of $(\theta_j)_{j\not=i}$, $u$ is differentiable at $\gamma(\sigma)$ and  is twice differentiable Lebesgue almost everywhere along the arc $\gamma$ and we can write
$$e(\gamma)=u(\gamma({3}))-u_\infty(\gamma({3}))-(u(\gamma(\sigma))-u_\infty(\gamma(\sigma)))=\int_\sigma^{3}(du(\gamma(t))-du_\infty(\gamma(t)))\dot\gamma(t)dt;$$
and also because $u-u_\infty$ is semi-concave and $\ddot\gamma=0$, we have\\
$e(\gamma)\leq$
\[
\begin{split}
&
\int_\sigma^{3}\bigg( \big( du(\gamma(\sigma))-du_\infty(\gamma(\sigma)) \big)\ \dot\gamma(\sigma)+\\
&\int_\sigma^t[(D^2u(\gamma(s))-D^2u_\infty(\gamma(s)))(\dot\gamma(s), \dot\gamma(s)) + (du(\gamma(s))-du_\infty(\gamma(s)))\ddot\gamma(s)  ]ds\bigg)dt
\end{split}
\]
i.e.
$$ e(\gamma)\leq({3}-\sigma)\big(du-du_\infty)(\gamma(\sigma) \big) \ \dot\gamma(\sigma)+\int_\sigma^{3} {(3-t)} (D^2u(\gamma(t))-D^2u_\infty(\gamma(t)))(\dot\gamma(t), \dot\gamma(t))dt,
$$
and then because $\dot\gamma(t)=e_i$ and then $\|\dot\gamma\|=1$
\begin{equation}\label{Eint}e(\gamma)-({3}-\sigma)\left[ (du-du_\infty)(\gamma({\sigma}))e_i+{ \frac{3}{2}}\varepsilon\right]  \leq\int_\sigma^{3} ({3} -t ) (D^2u(\gamma(t))-D^2u_\infty(\gamma(t))-\varepsilon{\bf 1})(e_i, e_i)dt.
\end{equation}

{
Let $E_1 =\{t\in [0,3]~:~( D^2u-D^2u_\infty) (\gamma(t)) - \varepsilon{\bf 1}<0\}$ and $E_2 = [0,3]\setminus E_1$. Because of the uniform semi-concavity of $u$ and because $u_\infty$ is $C^{1,1}$, there exists $K>0$ such that $( D^2u-D^2u_\infty) (\theta) \leq K {\bf 1} $  for any $\theta\in\T^d$ where $D^2u-D^2u_\infty$ exists. 
{We deduce
\[
\begin{split}
&e(\gamma)-({3}-\sigma)\left[ (du-du_\infty)(\gamma(\sigma))e_i+{ \frac{3}{2}}\varepsilon\right]\\
\leq& \int_{E_2\cap[\sigma, 3]}(3-t)(D^2u(\gamma(t))-D^2u_\infty(\gamma(t))-K{\bf 1})(e_i, e_i)dt+\frac{3K}{2}(3-\sigma)Leb(E_2)\\
&\quad +\int_{E_1\cap[\sigma, 3]} (3-t)(D^2u(\gamma(t))-D^2u_\infty(\gamma(t))-\varepsilon{\bf 1})(e_i, e_i)dt\\
\end{split}
\]

Hence as the functions that appears in the integrals are non-positive, we deduce also
\[
\begin{split}
&e(\gamma)-({3}-\sigma)\left[ (du-du_\infty)(\gamma(\sigma))e_i+{ \frac{3}{2}}\varepsilon\right]\\
\leq& \int_{E_2\cap[1, 2]} (D^2u(\gamma(t))-D^2u_\infty(\gamma(t))-K{\bf 1})(e_i, e_i)dt+\frac{3K}{2}(3-\sigma)Leb(E_2)\\
&\quad +\int_{E_1\cap[1, 2]} (D^2u(\gamma(t))-D^2u_\infty(\gamma(t))-\varepsilon{\bf 1})(e_i, e_i)dt.\\
\end{split}
\]

We introduce the notation $\cE=\{ \theta\in \mathbb{T}^d ~:~ D^2 u - D^2 u_\infty <  \varepsilon {\bf 1} \}$. By Lemma~\ref{D2Lebestimate}, we have $Leb(\T^d\backslash\cE)<\varepsilon$.
Integrating with respect to $(\theta_j)_{j\not=i}\in \T^{d-1}$, we finally obtain

\[
\begin{split}
&-{3} \left(\|u-u_\infty\|_\infty+\| u-u_\infty\|_{C^1}+ \frac{3}{2}\varepsilon \right)\\
\leq& \int_{\T^d\backslash\cE}  (D^2u(\theta)-D^2u_\infty(\theta)-K{\bf 1})(e_i, e_i)d\theta+\frac{3K}{2}(3-\sigma)Leb(\T^d\backslash\cE)\\
&\quad +\int_{\cE} (D^2u(\theta)-D^2u_\infty(\theta)-\varepsilon{\bf 1})(e_i, e_i)d\theta\\
\end{split}
\]

and then because $Leb(\T^d\backslash \cE)<\varepsilon$, we have
\[
\begin{split}
&-{3} \left(\|u-u_\infty\|_\infty+\| u-u_\infty\|_{C^1}+ \frac{3}{2}(K+1)\varepsilon \right)\\
\leq& \int_{\T^d\backslash\cE}  (D^2u(\theta)-D^2u_\infty(\theta)-K{\bf 1})(e_i, e_i)d\theta\\
&\quad +\int_{\cE}(D^2u(\theta)-D^2u_\infty(\theta)-\varepsilon{\bf 1})(e_i, e_i)d\theta.\\
\end{split}
\]
We deduce that
\begin{equation}\label{EsupK}
-\int_{\T^d\backslash\cE}  (D^2u(\theta)-D^2u_\infty(\theta)-K{\bf 1})(e_i, e_i)d\theta\leq {3} \left(\|u-u_\infty\|_\infty+\| u-u_\infty\|_{C^1}+ \frac{3}{2}(K+1)\varepsilon \right)
\end{equation}
and
\begin{equation}\label{Esupeps}
-\int_{\cE}  (D^2u(\theta)-D^2u_\infty(\theta)-\varepsilon{\bf 1})(e_i, e_i)d\theta\leq {3} \left(\|u-u_\infty\|_\infty+\| u-u_\infty\|_{C^1}+ \frac{3}{2}(K+1)\varepsilon \right).
\end{equation}

Let $v=\sum v_ie_i\in \R^d$ such that $\|v\|=1$. Then we have  for every $\theta\in\cE$
$$0\leq -(D^2u(\theta)-D^2u_\infty(\theta)-\varepsilon{\bf 1})(v, v)\leq -d^2\sup_{1\leq i\leq d}(D^2u(\theta)-D^2u_\infty(\theta)-\varepsilon{\bf 1})(e_i, e_i)$$
and for every $\theta\in\T^d\backslash \cE$
$$0\leq -(D^2u(\theta)-D^2u_\infty(\theta)-K{\bf 1})(v, v)\leq -d^2\sup_{1\leq i\leq d}(D^2u(\theta)-D^2u_\infty(\theta)-K{\bf 1})(e_i, e_i).$$
We deduce that
$$\int_\cE \| (D^2u-D^2u_\infty)(\theta)-\varepsilon{\bf 1}\|d\theta\leq-d^2\sum_{1\leq i\leq d}\int_\cE((D^2u-D^2u_\infty)(\theta)-\varepsilon{\bf 1})(e_i, e_i)d\theta$$
and
$$\int_{\T^d\backslash\cE} \| (D^2u-D^2u_\infty)(\theta)-K{\bf 1}\|d\theta\leq-d^2\sum_{1\leq i\leq d}\int_{\T^d\backslash\cE}((D^2u-D^2u_\infty)(\theta)-K{\bf 1})(e_i, e_i)d\theta.$$

Observe that $d_{2,1}(u, u_\infty)\leq$
$$\int_{\T^d\backslash\cE} \| (D^2u-D^2u_\infty)(\theta)-\varepsilon{\bf 1}\|d\theta+K.Leb(\T^d\backslash\cE)
+ \int_\cE \| (D^2u-D^2u_\infty)(\theta)-K{\bf 1}\|d\theta+\varepsilon.Leb(\cE)$$
and then $$d_{2,1}(u, u_\infty)\leq \int_{\T^d\backslash\cE} \| (D^2u-D^2u_\infty)(\theta)-K{\bf 1}\|d\theta
+ \int_\cE \| (D^2u-D^2u_\infty)(\theta)-\varepsilon{\bf 1}\|d\theta+(K+1)\varepsilon$$
{so $d_{2,1}(u, u_\infty)\leq$ $$ -d^2\sum_{1\leq i\leq d}\left( \int_{\T^d\backslash\cE}((D^2u-D^2u_\infty)(\theta)-K{\bf 1})(e_i, e_i)d\theta +
\int_\cE((D^2u-D^2u_\infty)(\theta)-\varepsilon{\bf 1})(e_i, e_i)d\theta\right)+(K+1)\varepsilon$$
}\\
Using Equations (\ref{EsupK}) and (\ref{Esupeps}), we deduce

$$d_{2,1}(u, u_\infty)\leq 6d^3\left(\|u-u_\infty\|_\infty+\| u-u_\infty\|_{C^1}+ {2}(K+1)\varepsilon \right).
$$

This is the wanted result.}

\bibliographystyle{alpha}
\bibliography{referencebis}
\end{document}